\documentclass[onefignum,onetabnum]{siamart171218}

\usepackage{lipsum}
\usepackage{amsfonts}
\usepackage{graphicx}
\usepackage{epstopdf}
\usepackage{algorithmic}
\ifpdf
  \DeclareGraphicsExtensions{.eps,.pdf,.png,.jpg}
\else
  \DeclareGraphicsExtensions{.eps}
\fi

\usepackage{array}
\usepackage{amsmath}
\usepackage{amssymb}
\usepackage{color}
\usepackage{bbm}
\usepackage{todonotes}

\usepackage[leqno]{amsmath}
\usepackage{etoolbox}

\usepackage[amsmath,thmmarks,hyperref]{ntheorem}

\patchcmd{\SetTagPlusEndMark}{$}{}{}{}
\patchcmd{\SetTagPlusEndMark}{$}{}{}{}

\newcommand{\R}{\mathbb{R}}

\newcommand{\mD}{\mathcal{D}}

\newcommand{\dd}{\, \text{d}}

\newsiamremark{remark}{Remark}
\newsiamremark{hypothesis}{Hypothesis}
\crefname{hypothesis}{Hypothesis}{Hypotheses}
\newsiamthm{claim}{Claim}

\headers{RHC Method for Linear-Quadratic Optimal Control Problems}{T. Breiten and L. Pfeiffer}

\title{On the Turnpike Property and the Receding-Horizon Method for Linear-Quadratic Optimal Control Problems\thanks{Submitted to the editors on November 8, 2018.}}

\author{Tobias Breiten\thanks{Institute of Mathematics and Scientific Computing, University of Graz, Austria 
  (\email{tobias.breiten@uni-graz.at}).}
\and Laurent Pfeiffer\thanks{Inria and CMAP (UMR 7641), CNRS, Ecole  Polytechnique,  Institut  Polytechnique  de Paris, Route de Saclay, 91128 Palaiseau, France (\email{laurent.pfeiffer@uni-graz.at}).}
}
\usepackage{amsopn}

\ifpdf
\hypersetup{
  pdftitle={On the Turnpike Property and the Receding-Horizon Method for Linear-Quadratic Optimal Control Problems},
  pdfauthor={T. Breiten and L. Pfeiffer}
}
\fi

\begin{document}

\maketitle

\begin{abstract}
Optimal control problems with a very large time horizon can be tackled with the Receding Horizon Control (RHC) method, which consists in solving a sequence of optimal control problems with small prediction horizon. The main result of this article is the proof of the exponential convergence (with respect to the prediction horizon) of the control generated by the RHC method towards the exact solution of the problem. The result is established for a class of infinite-dimensional linear-quadratic optimal control problems with time-independent dynamics and integral cost. Such problems satisfy the turnpike property: the optimal trajectory remains most of the time very close to the solution to the associated static optimization problem.
Specific terminal cost functions, derived from the Lagrange multiplier associated with the static optimization problem, are employed in the implementation of the RHC method.
\end{abstract}

\begin{keywords}
Receding horizon control, model predictive control, value function, optimality systems, Riccati equation, turnpike property.
\end{keywords}

\begin{AMS}
49J20,  49L20, 49Q12, 93D15.
\end{AMS}

\section{Introduction}

\subsection{Context}

We consider in this article the following class of linear-quadratic optimal control problems:
\begin{equation*} \label{prob:turnpike} \tag{$P$}
\left\{
\begin{array}{l}
{\displaystyle
\inf_{\begin{subarray}{c} y \in W(0,\bar{T}) \\ u \in L^2(0,\bar{T};U) \end{subarray}} \! \!
J_{\bar{T},Q,q}(u,y):= \! \int_0^{\bar{T}} \! \ell(y(t),u(t)) \dd t + {\textstyle \frac{1}{2}} \langle y(\bar{T}), Q y(\bar{T}) \rangle + \langle q, y(\bar{T}) \rangle, }\\[0.5em]
\text{subject to: } \dot{y}(t)= Ay(t) + Bu(t) + f^\diamond, \quad y(0)= y_0,
\end{array}
\right.
\end{equation*}
where the integral cost $\ell$ is defined by
\begin{equation*}
\ell(y,u)= {\textstyle \frac{1}{2} } \| Cy \|_Z^2 + \langle g^\diamond, y \rangle_{V^*,V}
+ {\textstyle \frac{\alpha}{2}} \| u\|^2_U + \langle h^\diamond, u \rangle_U.
\end{equation*}
Here $V \subset Y \subset V^*$ is a Gelfand triple of real Hilbert spaces \cite[page 147]{Tro10}, where the embedding of $V$ into $Y$ is dense, $V^*$ denotes the topological dual of $V$ and $U,Z$ denote further Hilbert spaces. The operator $A\colon \mD(A) \subset Y \rightarrow Y$ is the infinitesimal generator of an analytic $C_0$-semigroup $e^{At}$ on $Y$, $B \in \mathcal{L}(U,V^*)$, $C \in \mathcal{L}(Y,Z)$, $\alpha> 0$, $Q \in \mathcal{L}(Y)$ is self-adjoint positive semi-definite and $\mD(A)$ denotes the domain of $A$. The pairs $(A,B)$ and $(A,C)$ are assumed to be stabilizable and detectable, respectively. The elements $y_0 \in Y$, $f^\diamond \in V^*$, $g^\diamond \in V^*$, $h^\diamond \in U$, $q \in Y$ are given.

The following problem, referred to as \emph{static} optimization problem (or \emph{steady-state} optimization problem), has a unique solution $(y^\diamond,u^\diamond)$ with unique associated Lagrange multiplier $p^\diamond$:
\begin{equation} \label{eq:turnpike_pb_setub}
\inf_{(y,u) \in V \times U} \ell(y,u),
\quad \text{subject to: } Ay + Bu + f^\diamond = 0.
\end{equation}
A particularly important feature of \eqref{prob:turnpike} is the \emph{exponential turnpike} property. It states that there exist two constants $M>0$ and $\lambda>0$, independent of $\bar{T}$, such that for all $t \in [0,\bar{T}]$, $\| \bar{y}(t) - y^\diamond \|_Y \leq M \big( e^{-\lambda t} + e^{-\lambda(\bar{T}-t)} \big)$, where $\bar{y}$ denotes the optimal trajectory.
The trajectory $\bar{y}$ is thus made of three arcs, the first and last one being transient short-time arcs and the middle one a long-time arc, where the trajectory remains close to $y^\diamond$.
We refer the reader to the books \cite{Z06,Z19}, where different turnpike properties are established for different kinds of systems. We mention in particular the general characterization of the turnpike phenomenon for linear systems in \cite[Section 5.34]{Z19}. For linear-quadratic problems, we mention the articles \cite{DGSW14,GG18} for discrete-time systems and the articles \cite{PZ13} and \cite{PZ16} containing results for classes of infinite-dimensional systems. We also mention the early reference \cite{AL85} dealing with a tracking problem. Exponential turnpike properties have been established for non-linear systems in \cite{TZ15} and \cite{TZZ18}.

The aim of this article is to analyze the efficiency of the Receding Horizon Control (RHC) method (also called Model Predictive Control method), that we briefly present here, a detailed description can be found in Section \ref{section:rhc}.
We consider an implementation of the method with three parameters: a sampling time $\tau$, a prediction horizon $T$, and a prescribed number of iterations $N$. The method generates in a recursive way a control $u_{RH}$ and its associated trajectory $y_{RH}$. At the beginning of iteration $n$ of the algorithm, $u_{RH}$ and $y_{RH}$ have already been computed on $(0,n\tau)$. Then, an optimal control problem is solved on the interval $(n\tau, n\tau + T)$, with initial condition $y_{RH}(n\tau)$, with the same integral cost as in \eqref{prob:turnpike}, but with the following terminal cost function:
\begin{equation} \label{eq:def_phi_setup}
\phi(y)= \langle p^\diamond, y \rangle_Y.
\end{equation}
The restriction of the solution to $(n\tau,(n+1)\tau)$ is then concatenated with $(y_{RH},u_{RH})$. At iteration $N$, a last optimal control problem is solved on the interval $(N \tau,\bar{T})$.
The definition \eqref{eq:def_phi_setup} is actually a particular choice of the terminal cost among a general class of linear-quadratic functions.
For this specific definition, the main result of the article is the following estimate:
\begin{align}
& \max \big( \| y_{RH}- \bar{y} \|_{W(0,\bar{T})}, \| u_{RH} - \bar{u} \|_{L^2(0,\bar{T};U)} \big) \notag \\
& \qquad \qquad \leq M e^{-\lambda(T-\tau)} \big( e^{-\lambda T} \| y_0 - y^\diamond \|_Y + e^{-\lambda(\bar{T}- (N \tau + T))}  \| \tilde{q} \|_Y \big), \label{eq:main_setup}
\end{align}
with $\tilde{q}= q - p^\diamond + Qy^\diamond$. The estimate is proven for sampling times and prediction horizons satisfying $0 < \tau_0 \leq \tau \leq T \leq \bar{T}$. 
The constants $\tau_0>0$, $M>0$, and $\lambda >0$ are independent of $y_0$, $f^\diamond$, $g^\diamond$, $h^\diamond$, $q$, $N$, $\tau$, $T$, and $\bar{T}$.
Let us mention that the lower bound $\tau_0$ cannot be chosen arbitrarily small.
The idea of taking $\langle p^\diamond, y \rangle$ as a terminal cost has been proposed in the recent article \cite{ZF18} in the context of discrete-time problems. 

The choice of an appropriate terminal cost function is a key issue in the design of an appropriate RHC scheme. When $\phi$ is the exact value function, then the RHC method generates the exact solution to the problem, as a consequence of the dynamic programming principle. The article will give a (positive) answer to the following question: Does the RHC algorithm generate an efficient control if a good approximation of the value function is used as terminal cost function?
The construction of such an approximation is here possible thanks to the turnpike property.
We will see that the derivative of the value function (with respect to the initial condition), evaluated at $y^\diamond$, converges to $p^\diamond$ as $\bar{T} - t$ increases. Roughly speaking, the definition \eqref{eq:def_phi_setup} is a kind of first-order Taylor approximation of the value function, around $y^\diamond$.

The RHC method is receiving a tremendous amount of attention and it is frequently used in control engineering, in particular because it is computationally easier to solve a problem with short horizon. Another reason is that the method can be used as a feedback mechanism: when the control is computed in real time with the RHC method, perturbations having arisen in the past can be taken into account.
Let us point at some references from the large literature on receding horizon control. For finite-dimensional systems, we mention \cite{GP, MRRS}, for infinite-dimensional systems, we mention \cite{AK2, AK1, G}, and for discrete-time systems the articles \cite{GMTT, GR}.

In the current framework, the first-order optimality conditions take the usual form of a linear optimality system. The central idea for the derivation of estimate \eqref{eq:main_setup} is to compare the right-hand sides of the two optimality systems associated with the exact solution of \eqref{prob:turnpike} (restricted to $(n\tau,n \tau + T)$) and with the solution to the optimal control problem with short prediction horizon $T$. This comparison is realized with the help of a priori bounds for linear optimality systems in specific weighted spaces.
The analysis of the optimality systems is an important part of the present article. The a priori bounds that we have obtained are of general interest. A classical technique (used in particular in \cite{PZ13,TZ15}), allowing to decouple the optimality systems, plays an important role.

The article is structured as follows. In Section \ref{section:lin_opt_sys}, we prove our error bound in weighted spaces for the optimality systems associated with \eqref{prob:turnpike}. Some additional properties on linear optimality systems are provided in Section \ref{section:add_res}.
We formulate then the class of linear-quadratic problems to be analyzed in Section \ref{section:lq_pb}. The turnpike property and some properties of the value function are then established. Section \ref{section:rhc} deals with the RHC method and contains our main result (Theorem \ref{theo:main}). An extension to infinite-horizon problems is realized in Section \ref{section:infinite_horizon}. Finally, we provide numerical results showing the tightness of our error estimate in Section \ref{section:numerics}.

\subsection{Vector spaces}

For $T \in (0,\infty)$, we make use of the vector space $W(0,T) = \big\{ y \in L^2(0,T;V) \,|\, \dot{y} \in L^2 (0,T;V^*) \big\}$.
As it is well-known, $W(0,T)$ is continuously embedded in $C([0,T],Y)$. We can therefore equip it with the following norm:
\begin{equation*}
\| y \|_{W(0,T)}
= \max \big( \| y \|_{L^2(0,T;V)}, \| \dot{y} \|_{L^2(0,T;V^*)}, \| y \|_{L^\infty(0,T;Y)} \big).
\end{equation*}

\paragraph{Weighted spaces}

Let $\mu \in \R$ be given, let $T \in (0,\infty)$. We denote by $L_{\mu}^2(0,T;U)$ the space of measurable functions $u \colon (0,T) \rightarrow U$ such that
\begin{equation*}
\| u \|_{L_{\mu}^2(0,T;U)} := \| e^{\mu \cdot} u(\cdot) \|_{L^2(0,T;U)} = \Big( \int_0^T \| e^{\mu t} u(t) \|_U^2 \dd t \Big)^{1/2} < \infty.
\end{equation*}
Observing that the mapping
$u \in L_{\mu}^2(0,T;U) \mapsto e^{\mu \cdot} u \in L^2(0,T;U)$
is an isometry, we deduce that $L_{\mu}^2(0,T;U)$ is a Banach space. Since $e^{\mu \cdot}$ is bounded from above and from below by a positive constant, we have that for all measurable $u \colon (0,T) \rightarrow U$, $u \in L^2(0,T;U)$ if and only if $u \in L^2_\mu(0,T;U)$. The spaces $L^2(0,T;U)$ and $L^2_\mu(0,T;U)$ are therefore the same vector space, equipped with two different norms. We define in a similar way the space $L_\mu^2(0,T;X)$, for a given Hilbert space $X$.
Similarly, we define the space $L_\mu^\infty(0,T;Y)$ of measurable mappings from $y \colon (0,T) \rightarrow Y$ such that
\begin{equation*}
\| y \|_{L_\mu^\infty(0,T;Y)} := \| e^{\mu \cdot} y(\cdot) \|_{L^\infty(0,T;Y)} < \infty.
\end{equation*}
We finally define the Banach space $W_\mu(0,T)$ as the space of measurable mappings $y \colon (0,T) \rightarrow V$ such that $e^{\mu \cdot} y \in W(0,T)$. One can check that for all measurable mappings $y \colon (0,T) \rightarrow V$, $y \in W(0,T)$ if and only if $y \in W_\mu(0,T)$.

For $T \in (0,\infty)$ and $\mu \in \R$, we introduce the space
\begin{equation} \label{eq:spaceLambda}
\Lambda_{T,\mu}= W_{\mu}(0,T) \times L_{\mu}^2(0,T;U) \times W_{\mu}(0,T),
\end{equation}
equipped with the norm $\| (y,u,p) \|_{\Lambda_{T,\mu}} = \max \big( \| y \|_{W_\mu(0,T)}, \| u \|_{L_\mu^2(0,T;U)}, \| p \|_{W_\mu(0,T)} \big)$.
For $T \in (0,\infty)$, we define the space
\begin{equation} \label{eq:spaceOmega}
\Upsilon_{T,\mu}= Y \times L_{\mu}^2(0,T;V^*) \times L_{\mu}^2(0,T;V^*) \times L_{\mu}^2(0,T;U) \times Y
\end{equation}
that we equip with the norm
\begin{equation*}
\| (y_0,f,g,h,q) \|_{\Upsilon_{T,\mu}}
= \max \big( \| y_0 \|_Y, \| (f,g,h) \|_{L_{\mu}^2(0,T;V^*\times V^* \times U)}, e^{\mu T} \| q \|_{Y} \big).
\end{equation*}
Let us emphasize the fact that the component $q$ appears with a weight $e^{\mu T}$ in the above norm. The spaces $\Lambda_{T,0}$ and $\Lambda_{T,\mu}$ (resp.\@ $\Upsilon_{T,0}$ and $\Upsilon_{T,\mu}$) are the same vector space, equipped with two different norms. In the following lemma, the equivalence between these two norms is quantified.

\begin{lemma} \label{lemma:embedding}
For all $\mu_0$ and $\mu_1$ with $\mu_0 \leq \mu_1$, there exists a constant $M>0$ such that for all $T$, for all $(y,u,p) \in \Lambda_{T,0}$,
\begin{align*}
\| (y,u,p) \|_{\Lambda_{T,\mu_0}} \leq \ & M \| (y,u,p) \|_{\Lambda_{T,\mu_1}}, \\
\| (y,u,p) \|_{\Lambda_{T,\mu_1}} \leq \ & M e^{(\mu_1-\mu_0)T} \| (y,u,p) \|_{\Lambda_{T,\mu_0}},
\end{align*}
and such that, similarly, for all $(y_0,f,g,h,q) \in \Upsilon_{T,0}$,
\begin{align*}
\| (y_0,f,g,h,q) \|_{\Upsilon_{T,\mu_0}} \leq \ & M \| (y_0,f,g,h,q) \|_{\Upsilon_{T,\mu_1}}, \\
\| (y_0,f,g,h,q) \|_{\Upsilon_{T,\mu_1}} \leq \ & M e^{(\mu_1-\mu_0)T} \| (y_0,f,g,h,q) \|_{\Upsilon_{T,\mu_0}}.
\end{align*}
\end{lemma}

\begin{proof}[Proof of Lemma \ref{lemma:embedding}]
Let $y \in W(0,T)$ and $u \in L^2(0,T;U)$. For proving the lemma, it suffices to prove the existence of $M>0$, independent of $T$, $y$, and $u$, such that
\begin{equation} \label{lemma:embedding_u}
\| u \|_{L_{\mu_0}^2(0,T;U)} \leq M \| u \|_{L_{\mu_1}^2(0,T;U)}, \quad
\| u \|_{L_{\mu_1}^2(0,T;U)} \leq M e^{(\mu_1-\mu_0)T} \| u \|_{L_{\mu_0}^2(0,T;U)},
\end{equation}
and such that
\begin{equation} \label{lemma:embedding_y}
\| y \|_{W_{\mu_0}(0,T)} \leq M \| y \|_{W_{\mu_1}(0,T)}, \quad
\| y \|_{W_{\mu_1}(0,T)} \leq M e^{(\mu_1-\mu_0)T} \| y \|_{W_{\mu_0}(0,T)}.
\end{equation}
The inequalities \eqref{lemma:embedding_u} can be easily verified (with $M=1$). One can also easily verify that
\begin{align*}
\| y \|_{L_{\mu_0}^2(0,T;V)} \leq M \| y \|_{L_{\mu_1}^2(0,T;V)}, \quad
\| y \|_{L_{\mu_1}^2(0,T;V)} \leq M e^{(\mu_1-\mu_0)T} \| y \|_{L_{\mu_0}^2(0,T;V)} \\
\| y \|_{L_{\mu_0}^\infty(0,T;Y)} \leq M \| y \|_{L_{\mu_1}^\infty(0,T;Y)}, \quad
\| y \|_{L_{\mu_1}^\infty(0,T;Y)} \leq M e^{(\mu_1-\mu_0)T} \| y \|_{L_{\mu_0}^\infty(0,T;Y)}.
\end{align*}
Let $z_0(t)= e^{\mu_0 t} y(t)$ and $z_1(t)= e^{\mu_1 t} y(t)$. For proving \eqref{lemma:embedding_y}, it remains to compare $\| \dot{z}_0 \|_{L^2(0,T;V^*)}$ and $\| \dot{z}_1 \|_{L^2(0,T;V^*)}$. We have
$z_0(t)= e^{(\mu_0-\mu_1)t} z_1(t)$ and thus
$\dot{z}_0(t)= (\mu_0-\mu_1) z_0(t) + e^{(\mu_0-\mu_1)t} \dot{z}_1(t)$.
We deduce that
\begin{equation*}
\| \dot{z}_0 \|_{L^2(0,T;V^*)}
\leq M \| z_0 \|_{L^2(0,T;V)} + \| \dot{z}_1 \|_{L^2(0,T;V^*)}
\leq M \| z_1 \|_{W(0,T)} = M \| y \|_{W_{\mu_1}(0,T)}.
\end{equation*}
Similarly, we have $\dot{z}_1(t)= (\mu_1-\mu_0) z_1(t) + e^{(\mu_1-\mu_0)t} \dot{z}_0(t)$.
We deduce that
\begin{align*}
\| \dot{z}_1 \|_{L^2(0,T;V^*)}
\leq \ & M \| z_1 \|_{L^2(0,T;V)} + e^{(\mu_1-\mu_0)T} \| \dot{z}_0 \|_{L^2(0,T;V^*)} \\
\leq \ & M e^{(\mu_1-\mu_0)T} \| z_0 \|_{W(0,T)} \\
= \ & M e^{(\mu_1-\mu_0)T} \| y \|_{W_{\mu_0}(0,T)}.
\end{align*}
The inequalities \eqref{lemma:embedding_y} follow. This concludes the proof.
\end{proof}

\subsection{Assumptions}

Throughout the article we assume that the following four assumptions hold true.
\begin{itemize}
\item[(A1)] The operator $-A$ can be associated with a $V$-$Y$ coercive bilinear
form $a\colon V\times V \to \mathbb R$ which is such that there exist $\lambda_0 > 0$ and $\delta \in \R$ satisfying $a(v,v) \geq \lambda_0 \| v \| _V^2 - \delta \| v\|_Y^2$, for all $v \in V$.
\item[(A2)] [Stabilizability] There exists an operator $F \in \mathcal{L}(Y,U)$ such that the semigroup $e^{(A+BF) t}$ is exponentially stable on $Y$.
\item[(A3)] [Detectability] There exists an operator $K \in \mathcal{L}(Z,Y)$ such that the semigroup $e^{(A-KC) t}$ is exponentially stable on $Y$.
\end{itemize}
Assumptions (A2) and (A3) are well-known and analysed for infinite-dimensional systems, see e.g.\@ \cite{CurZ95}.
Consider the algebraic Riccati equation: for all $y_1$ and $y_2 \in \mathcal{D}(A)$,
\begin{equation} \label{eq:are}
\langle A^* \Pi y_1,y_2 \rangle_Y + \langle \Pi A y_1,y_2 \rangle_Y + \langle
Cy_1,Cy_2\rangle_Z - {\textstyle \frac{1}{\alpha} }  \langle B^* \Pi y_1,B^* \Pi y_2\rangle_U= 0.
\end{equation}
Due to the (exponential) stabilizability and detectability assumptions, it is well-known (see \cite[Theorem 6.2.7]{CurZ95} and \cite[Theorem 2.2.1]{LasT00}) that \eqref{eq:are} has a unique nonnegative self-adjoint solution $\Pi \in \mathcal{L}(Y,V) \cap \mathcal{L}(V^*,Y)$. 
Additionally, the semigroup generated by the operator $A_\pi:=A-\frac{1}{\alpha}BB^*\Pi$ is exponentially stable on $Y$.
We fix now, for the rest of the article, a real number $\lambda$ such that
\begin{align}\label{eq:spect_absc}
0 < \lambda < \bar{\lambda}:= -{\text{sup}}_{\mu \in \sigma(A_\pi)} \mathrm{Re}(\mu).
\end{align}
With (A1) holding the operator $A$ associated with the form $a$ generates an analytic semigroup that we denote by $e^{At}$, see e.g.\@ \cite[Sections 3.6 and 5.4]{Tan79}. Let us set $A_0=A-\lambda_0 I$. Then $-A_0$ has a bounded inverse in $Y$, see
\cite[page 75]{Tan79}, and in particular it is maximal accretive, see
\cite{Tan79}. We have $\mD(A_0)= \mD(A)$ and the fractional powers of
$-A_0$ are well-defined.
In particular,
$\mD((-A_0)^{\frac{1}{2}})=[\mD(-A_0),Y]_{\frac{1}{2}}:=(\mD(-A_0),Y)_
{\frac{1}{2},2}$ the real interpolation space with indices 2 and $\frac{1}{2}$,
see \cite[Proposition 6.1, Part II, Chapter 1]{Benetal07}.
Assumption (A4) below will only be used in the proof Lemma \ref{lemma:existence_turnpike}, where the existence and uniqueness of a solution $(y^\diamond,u^\diamond)$ to the static problem is established. It is not necessary for the analysis of optimality systems done in Sections \ref{section:lin_opt_sys} and \ref{section:add_res}.

\begin{itemize}
\item[(A4)] It holds that $[\mathcal{D}(-A_0),Y]_{\frac{1}{2}} = [\mathcal{D}(-A_0^*),Y]_{\frac{1}{2}\phantom{,}} = V$.
\end{itemize}

\section{Linear optimality systems} \label{section:lin_opt_sys}

The section is dedicated to the analysis of the following optimality system:
\begin{equation} \label{eq:non_reg_os}
\begin{cases}
\begin{array}{rll}
y(0) = & \! \! \! y_0 \qquad & \text{in $Y$} \\
\dot{y}-(Ay + Bu) = & \! \! \! f \qquad & \text{in $L_\mu^2(0,T;V^*)$} \\
-\dot{p} - A^* p - C^* C y = & \! \! \! g & \text{in $L_\mu^2(0,T;V^*)$} \\
\alpha u + B^*p  = & \! \! \! - h & \text{in $L_\mu^2(0,T;U)$} \\
p(T) - Q y(T)= & \! \! \! q & \text{in $Y$},
\end{array}
\end{cases}
\end{equation}
where $\mu \in \{ -\lambda,0,\lambda \}$, $T>0$, $Q \in \mathcal{L}(Y)$ is self-adjoint and positive semi-definite, and $(y_0,f,g,h,q) \in \Upsilon_{T,\mu}$. Given two times $t_1 < t_2$, we introduce the operator $\mathcal{H} \colon W(t_1,t_2) \times L^2(t_1,t_2;U) \times W(t_1,t_2) \rightarrow L^2(t_1,t_2;V^*\times V^* \times U)$, defined by
\begin{equation*}
\mathcal{H} (y,u,p)
= \big(\dot{y}-(Ay + Bu),\
-\dot{p} - A^* p - C^* C y,\
\alpha u + B^*p \big).
\end{equation*}
The dependence of $\mathcal{H}$ with respect to $t_1$ and $t_2$ is not indicated and the underlying values of $t_1$ and $t_2$ are always clear from the context. The operator $\mathcal{H}$ enables us to write the three intermediate equations of \eqref{eq:non_reg_os} in the compact form $\mathcal{H}(y,u,p)= (f,g,-h)$.

The main result of the section is the following theorem, which is proved in subsection \ref{subsection:general}.

\begin{theorem} \label{theo:non_reg_os}
Let $\mathcal{Q} \subset \mathcal{L}(Y)$ be a bounded set of self-adjoint and positive semi-definite operators. For all $T>0$, for all $Q \in \mathcal{Q}$, for all $(y_0,f,g,h,q) \in \Upsilon_{T,0}$, there exists a unique solution $(y,u,p)$ to system \eqref{eq:non_reg_os}. Moreover, for all $\mu \in \{ - \lambda, 0, \lambda \}$, there exists a constant $M$ independent of $T$, $Q$, and $(y_0,f,g,h,q)$ such that
\begin{equation} \label{eq:estimate_non_reg_os}
\| (y,u,p) \|_{\Lambda_{T,\mu}}
\leq M \| (y_0,f,g,h,q) \|_{\Upsilon_{T,\mu}}.
\end{equation}
\end{theorem}

\begin{remark}
The result of the theorem, for $\mu=0$, is rather classical in the literature and can be established by analyzing the associated optimal control problem (see Lemma \ref{lemma:lq_pb}). The main novelty of our result is the estimate \eqref{eq:estimate_non_reg_os} in weighted spaces, with a constant $M$ which is independent of $T$. Let us mention that a similar result has been obtained in \cite[Theorem 3.1]{GSS19}, for negative weights. The proof is based on a Neumann-series argument.
Let us mention that the range of admissible weights in that reference is different from ours (compare in particular with \cite[Corollary 3.16]{GSS19}).
\end{remark}

\subsection{Decouplable optimality systems}

We prove in this subsection Theorem \ref{theo:non_reg_os} in the case where $Q= \Pi$ (Lemma \ref{lemma:linSysFiniteHorizon}).
We begin with a useful result on forward and backward linear systems with a right-hand side in $L_{\mu}^2(0,T;V^*)$ (Lemma \ref{lemma:exp_stab_forward_sys}).

\begin{lemma} \label{lemma_stab_a_pi}
For all $\mu \leq \lambda$, $A_\pi + \mu I$ generates an exponentially stable semigroup.
For all $\mu \geq -\lambda$, $A_\pi^* - \mu I$ generates an exponentially stable semigroup.
\end{lemma}

\begin{proof}
Let $\tilde{\lambda} \in (\lambda,\bar{\lambda})$.
Since the semigroup $e^{A_\pi t}$ is analytic, the \textit{spectrum determined growth condition} is satisfied, see e.g.\@ \cite{Tri75}. Hence, $\| e^{A_\pi t} \|_{\mathcal{L}(Y)} \le M e^{-\tilde{\lambda} t}$, where $M$ does not depend on $t$. Therefore,
$\| e^{(A_\pi + \mu I)t} \|_{\mathcal{L}(Y)} \leq M e^{(-\tilde{\lambda} + \mu)t}$, which proves the exponential stability of $A_\pi + \mu I$ since $-\tilde{\lambda} + \mu < -\lambda + \mu \leq 0$. Moreover, $(e^{(A_\pi + \mu I) t })^* = e^{(A_\pi + \mu I)^*t}$ (see \cite[page 41]{Paz83}), thus the operator $A_\pi^* - \mu I$ generates a exponentially stable semigroup as well, for $\mu \ge -\lambda$.
\end{proof}

\begin{lemma} \label{lemma:exp_stab_forward_sys}
For all $\mu \leq \lambda$, for all $T \in (0,\infty)$, for all $y_0 \in Y$, for all $f \in L_{\mu}^2(0,T;V^*)$, the following system:
\begin{equation} \label{eq_forward_sys}
\dot{y}= A_\pi y + f, \quad y(0)= y_0
\end{equation}
has a unique solution in $W_\mu(0,T)$. Moreover, there exists a constant $M>0$ independent of $T$, $y_0$, and $f$ such that $\| y \|_{W_\mu(0,T)} \leq M \big( \| y_0 \|_Y + \| f \|_{L_{\mu}^2(0,T;V^*)} \big)$.

For all $\mu \geq - \lambda$, for all $T \in (0,\infty)$, for all $q \in Y$, for all $\Phi \in L_{\mu}^2(0,T;V^*)$, the following system: $-\dot{r}= A_\pi^* r + \Phi$, $r(T)= q$ has a unique solution in $W_\mu(0,T)$. Moreover, there exists a constant $M>0$ independent of $T$, $q$, and $\Phi$ such that $\| r \|_{W_\mu(0,T)} \leq M \big( \| \Phi \|_{L_{\mu}^2(0,T;V^*)} + e^{\mu T} \| q \|_Y \big)$.
\end{lemma}

\begin{proof}
Let us prove the first statement. Let $y \in W(0,T)$.
Defining $y_\mu: = e^{\mu \cdot} y \in W_\mu(0,T)$ and $f_\mu := e^{\mu \cdot}f \in L_{\mu}^2(0,T;V^*)$, we observe that $y$ solves \eqref{eq_forward_sys} if and only if $y_\mu$ is the solution to the following system:
\begin{equation} \label{eq_forward_sys_shifted}
\dot{y}_{\mu} = (A_\pi + \mu I) y_{\mu} + f_\mu, \quad y_\mu(0) = y_0.
\end{equation}
Since $\mu \leq \lambda$, the operator $A_\pi + \mu I$ generates an exponentially stable semigroup, by Lemma \ref{lemma_stab_a_pi}. Standard regularity results for analytic semigroups ensure the existence and uniqueness of a solution to \eqref{eq_forward_sys_shifted}, as well as the existence of a constant $M > 0$ independent of $T$, $y_0$, and $f$ such that $\| y_\mu \|_{W(0,T)} \leq M \big( \| y_0 \| + \| f_\mu \|_{L^2(0,\infty)} \big)$, which is the estimate that was to be proved.

The second statement can be proved similarly with a time-reversal argument.
\end{proof}

We are now ready to analyze \eqref{eq:non_reg_os} in the case where $Q= \Pi$. The key idea is to decouple the system with the help of the variable $r= p- \Pi y$. This variable is indeed the solution to a backward differential equation which is independent of $y$, $u$, and $p$. Let us mention that this remarkable property only holds in the case $Q= \Pi$.

\begin{lemma} \label{lemma:linSysFiniteHorizon}
For all $\mu \in [-\lambda,\lambda]$, for all $T>0$, for all $(y_0,f,g,h,q) \in \Upsilon_{T,\mu}$, there exists a unique $(y,u,p) \in \Lambda_{T,\mu}$ solution to \eqref{eq:non_reg_os} with $Q= \Pi$.
Moreover, there exists a constant $M>0$, independent of $T$ and $(y_0,f,g,h,q)$ such that
\begin{equation} \label{eq:estimateForSysLin}
\| (y,u,p) \|_{\Lambda_{T,\mu}}
\leq M \| (y_0,f,g,h,q) \|_{\Upsilon_{T,\mu}}.
\end{equation}
\end{lemma}

\begin{remark}
All along the article, the variable $M$ is a positive constant whose value may change from an inequality to the next one. When an estimate involving a constant $M$ independent of some variables (for example $T$) has to be proved, then all constants $M$ used in the corresponding proof are also independent of these variables.
\end{remark}

\begin{proof}[Proof of Lemma \ref{lemma:linSysFiniteHorizon}]
Let $\Phi \in L_{\mu}^2(0,T;V^*)$ be defined by $\Phi= \Pi f - \frac{1}{\alpha}\Pi B h + g$.
Let us denote by $r \in W_{\mu}(0,T)$ the unique solution to the system $-\dot{r} = A_\pi ^* r + \Phi$, $t \in [0,T)$, $r(T)= q$.
By Lemma \ref{lemma:exp_stab_forward_sys}, there exists a constant $M$, independent of $T$ and $(y_0,f,g,h,q)$ such that
\begin{equation} \label{eq:estimateOnR}
\| r \|_{W_{\mu}(0,T)} \leq M \big( \| \Phi \|_{L_{\mu}^2(0,T;V^*)} + e^{\mu T} \| q \|_Y \big) 
\leq M \| (y_0,f,g,h,q) \|_{\Upsilon_{T,\mu}}.
\end{equation}
By Lemma \ref{lemma:exp_stab_forward_sys}, the following system has a unique solution $y \in W_{\mu}(0,T)$: 
\begin{equation} \label{eq:def_y}
\dot{y}=  A_\pi y - {\textstyle \frac{1}{\alpha}} Bh + f - {\textstyle \frac{1}{\alpha}} BB^*r, \quad y(0)= y_0.
\end{equation}
Since $\big\| - {\textstyle \frac{1}{\alpha}} Bh + f - {\textstyle \frac{1}{\alpha}} BB^*r \big\|_{L_\mu^2(0,T;V^*)} \leq M \| (y_0,f,g,h,q) \|_{\Upsilon_{T,\mu}}$, we have that
\begin{equation} \label{eq:estimateOnY}
\| y \|_{W_{\mu}(0,T)}
\leq M \| (y_0,f,g,h,q) \|_{\Upsilon_{T,\mu}}.
\end{equation}
Let us set $p = \Pi y + r$.
Since $\Pi \in \mathcal{L}(Y,V) \cap \mathcal{L}(V^*,Y)$, we have that $\Pi y \in L^2(0,T;V)\cap H^1(0,\infty;Y)$. Therefore, using \eqref{eq:estimateOnR} and \eqref{eq:estimateOnY}, we obtain that $p \in W_{\mu}(0,T)$ with $\| p \|_{W_{\mu}(0,T)} \leq M \| (y_0,f,g,h,q) \|_{\Upsilon_{T,\mu}}$.
We finally define $u= - \frac{1}{\alpha}(h+B^*p )$.
We deduce from the estimate on $p$ that
$\| u \|_{L_{\mu}^2(0,T;U)} \leq M \| (y_0,f,g,h,q) \|_{\Upsilon_{T,\mu}}$. The bound \eqref{eq:estimateForSysLin} is proved.

Let us check that $(y,u,p)$ is a solution to the linear system \eqref{eq:non_reg_os}. It follows from the definition of $u$ that $\alpha u + B^*p= -h$. Using $p= \Pi y + r$ and \eqref{eq:def_y}, we obtain that
\begin{align*}
Ay + Bu + f= \ & Ay - {\textstyle \frac{1}{\alpha}} Bh - {\textstyle \frac{1}{\alpha}} BB^*p + f \\
= \ & Ay - {\textstyle \frac{1}{\alpha}} Bh - {\textstyle \frac{1}{\alpha}} BB^* \Pi y - {\textstyle \frac{1}{\alpha}} BB^*r + f
= \dot{y}.
\end{align*}
It remains to verify that the adjoint equation is satisfied. We obtain with the definitions of $p$, $y$, $\Phi$, and $A_\pi$ that $p(T)-\Pi y(T)= q$ and that
\begin{equation*}
\dot{p}
= \Pi \dot{y}+\dot{r} \\
= \Pi A_\pi y + \underbrace{\big( - {\textstyle \frac{1}{\alpha}} \Pi Bh + \Pi f \big)}_{= \Phi- g} + \underbrace{ \big(-{\textstyle \frac{1}{\alpha}} \Pi BB^*  \big)}_{= A_\pi^* - A^*}r + \dot{r}.
\end{equation*}
Using $\dot{r} + A_\pi^* r + \Phi= 0$ and \eqref{eq:are}, we obtain that
\begin{align*}
\dot{p}= \ & \Pi A_\pi y + \Phi - g + A_\pi^*r - A^* r + \dot{r} \\
 = \ & \Pi A_\pi y -A^* r-g
 = \big(\Pi A - {\textstyle \frac{1}{\alpha}} \Pi BB^*\Pi \big)y - A^*r - g \\
 = \ & -\big(A^* \Pi + C^* C \big)y - A^*r - g
 = - A^* p - C^*C y - g.
\end{align*}
Therefore, the adjoint equation is satisfied and $(y,u,p)$ is a solution to \eqref{eq:estimate_non_reg_os}.

It remains to show uniqueness. To this end, it suffices to consider the case where $(y_0,f,g,h,q)= (0,0,0,0,0)$. Let $(y,u,p)$ be a solution to \eqref{eq:non_reg_os}. Let $r= p-\Pi y$. One can easily see that $-\dot{r}= A_\pi^* r$, $r(T)= 0$, thus $r=0$. Then, one has to check that $y$ satisfies \eqref{eq:def_y}, with $y_0=0$, $f= 0$, $h=0$, and $r= 0$. Therefore, $y=0$. Finally, we obtain that $p= r + \Pi y= 0$ and that $u= -\frac{1}{\alpha}(h + B^*p)= 0$. Uniqueness is proved.
%
%
\end{proof}

\subsection{General case} \label{subsection:general}

We give a proof of Theorem \ref{theo:non_reg_os} in this subsection. We consider successively the cases $\mu= 0$, $\mu= -\lambda$, and $\mu= \lambda$.

\subsubsection{Case without weight}

Theorem \ref{theo:non_reg_os}, in the case where $\mu= 0$, can be established by analyzing the optimal control problem associated with \eqref{eq:non_reg_os}. This is the result of Lemma \ref{lemma:lq_pb} below. The proof is classical and uses very similar arguments to the ones used in \cite[Proposition 3.1]{BreKP18a}. 

We begin with a classical lemma, following from the detectability assumption.

\begin{lemma} \label{lemma:detectability}
There exists a constant $M>0$ such that for all $T >0$, for all $y_0 \in Y$, for all $u \in L^2(0,T;U)$, for all $f \in L^2(0,T;V^*)$, the solution $y \in W(0,T)$ to the system
\begin{equation*}
\dot{y}= Ay + Bu + f, \quad y(0)= y_0
\end{equation*}
satisfies the following estimate:
\begin{equation*}
\| y \|_{W(0,T)} \leq M \big( \| y_0 \|_Y + \| u \|_{L^2(0,T;U)} + \| f \|_{L^2(0,T;V^*)} + \| Cy \|_{L^2(0,T;Z)} \big).
\end{equation*}
\end{lemma}

\begin{proof}
Let $z \in W(0,T)$ be the solution to
\begin{equation*}
\dot{z}= Az + Bu + f + KC(y-z), \quad z(0)= y_0,
\end{equation*}
where $K$ is given by Assumption (A3). The above system can be re-written as follows:
\begin{equation*}
\dot{z}= (A-KC) z + Bu + f + KCy, \quad z(0)= y_0.
\end{equation*}
Since $(A-KC)$ is exponentially stable, there exists a constant $M$, independent of $T$, $y_0$, $u$, $f$, and $y$ such that
\begin{align}
\| z \|_{W(0,T)} \leq \ & M \big( \| y_0 \|_Y + \| Bu + f + KCy \|_{L^2(0,T;V^*)} \big) \notag \\
\leq \ & M \big( \| y_0 \|_Y + \| u \|_{L^2(0,T;U)} + \| f \|_{L^2(0,T;V^*)} + \| Cy \|_{L^2(0,T;Z)} \big).
\label{eq:obs_1}
\end{align}
Observing that $e:= z-y$ is the solution to $\dot{e}= (A-KC)e$, $e(0)=0$, we obtain that $e=0$ and that $z=y$. Thus $y$ satisfies \eqref{eq:obs_1}, as was to be proved.
\end{proof}

\begin{lemma} \label{lemma:lq_pb}
For all $T >0$, for all $Q \in \mathcal{Q}$, for all $(y_0,f,g,h,q) \in \Upsilon_{T,0}$, the following optimal control problem
\begin{equation*} \label{eq:lin_quad} \tag{$LQ$}
\left\{
\begin{array}{l}
{\displaystyle \inf_{ \begin{subarray}{c} y \in W(0,T) \\ u \in L^2(0,T;U) \end{subarray}}
\Big[ \int_0^T {\textstyle \frac{1}{2}} \| Cy(t) \|_{Z}^2 + \langle g(t),y(t) \rangle + {\textstyle \frac{\alpha}{2} } \|u(t)\|_U^2 + \langle h(t),u(t) \rangle_U \dd t } \\[-0.5em]
\qquad \qquad \qquad \qquad \qquad \qquad {\displaystyle + {\textstyle \frac{1}{2}} \langle y(T), Q y(T) \rangle_Y + \langle q, y(T) \rangle_Y \Big], } \\[1em]
\qquad \text{subject to: } \dot{y}= Ay + Bu + f, \quad y(0)= y_0,
\end{array}
\right.
\end{equation*}
has a unique solution $(y,u)$. There exists a unique associated adjoint variable $p$, which is such that $(y,u,p)$ is the unique solution to \eqref{eq:non_reg_os}. Moreover, there exists a constant $M$, independent of $T$, $Q$, and $(y_0,f,g,h,q)$ such that
\begin{equation} \label{eq:estim_non_weight}
\| (y,u,p) \|_{\Lambda_{T,0}} \leq M \| (y_0,f,g,h,q) \|_{\Upsilon_{T,0}}.
\end{equation}
\end{lemma}

\begin{proof}
We follow the same lines as in \cite[Lemma 3.2]{BreKP18a}.
Let us first bound the value of the problem. Let $y \in W(0,T)$ be the solution to
\begin{equation*}
\dot{y}= (A +BF)y + f, \quad y(0)= y_0,
\end{equation*}
where $F$ is given by Assumption (A2).
Since $(A +BF)$ is exponentially stable, there exists a constant $M$ such that
\begin{equation*}
\| y \|_{W(0,T)} \leq M \max \big( \| y_0 \|_Y, \| f \|_{L^2(0,T;V^*)} \big).
\end{equation*}
Let us set $u= Fy$. We have $\| u \|_{L^2(0,T;U)} \leq M \max \big( \| y_0 \|_Y, \| f \|_{L^2(0,T;V^*)} \big)$. Then, one can easily check the existence of a constant $M$ such that
\begin{equation*}
J_1(u,y) \leq M \| (y_0,f,g,h,q) \|_{\Upsilon_{T,0}}^2.
\end{equation*}
Now, we prove the existence of a solution to the problem. Let $(y_n,u_n)_{n \in \mathbb{N}} \in W(0,T) \times L^2(0,T;U)$ be a minimizing sequence such that for all $n \in \mathbb{N}$,
\begin{equation*}
J_1(y_n,u_n) \leq M \| (y_0,f,g,h,q) \|_{\Upsilon_{T,0}}^2.
\end{equation*}
We now look for a lower bound for $J_1$, so that we can further obtain a bound on $(y_n,u_n)$. We have
\begin{align*}
J_1(y_n,u_n) \geq \ & \frac{1}{2} \| Cy_n \|_{L^2(0,T;Z)}^2 - \| g \|_{L^2(0,T;V^*)} \| y_n \|_{W(0,T)} \\
& \qquad + \frac{\alpha}{2} \| u_n \|_{L^2(0,T;U)}^2 - \| h \|_{L^2(0,T;U)} \| u_n \|_{L^2(0,T;U)} - \| q \|_Y \| y_n(T) \|_Y \\
\geq \ & \frac{1}{2} \| Cy_n \|_{L^2(0,T;Z)}^2 + \frac{\alpha}{2} \Big( \| u_n \|_{L^2(0,T;U)}^2 - \frac{\| h \|_{L^2(0,T;U)}}{\alpha} \Big)^2
- \frac{\| h \|_{L^2(0,T;U)}^2}{2 \alpha} \\
& \qquad
- \frac{1}{2\varepsilon} \big( \| g \|_{L^2(0,T;V^*)} + \| q \|_Y \big)^2
- \frac{\varepsilon}{2} \| y_n \|_{W(0,T)}^2.
\end{align*}
Therefore, there exists a constant $M$ such that
\begin{align}
\| Cy_n \| \leq M \Big( \| (y_0,f,g,h,q) \|_{\Upsilon_{T,0}} + \frac{1}{\sqrt{\varepsilon}} \big( \| g \|_{L^2(0,T;V^*)} + \| q \|_Y \big) + \sqrt{\varepsilon} \| y_n \|_{W(0,T)} \Big), \label{eq:estimate_min_seq_1} \\
\| u_n \|  \leq M \Big( \| (y_0,f,g,h,q) \|_{\Upsilon_{T,0}} + \frac{1}{\sqrt{\varepsilon}} \big( \| g \|_{L^2(0,T;V^*)} + \| q \|_Y \big) + \sqrt{\varepsilon} \| y_n \|_{W(0,T)} \Big). \label{eq:estimate_min_seq_2}
\end{align}
Applying Lemma \ref{lemma:detectability} and estimate \eqref{eq:estimate_min_seq_1}, we obtain that
\begin{align*}
\| y_n \|_{W(0,T)} \leq \ & M \big( \| y_0 \|_Y + \| f \|_{L^2(0,T;V^*)} + \| u_n \|_{L^2(0,T;U)} + \| C y_n \|_{L^2(0,T;Z)} \big) \\
\leq \ & M \Big( \| (y_0,f,g,h,q) \|_{\Upsilon_{T,0}} + \sqrt{\varepsilon} \| y_n \|_{W(0,T)} + \frac{1}{\sqrt{\varepsilon}} \big( \| g \| + \| q \|_Y \big) \Big).
\end{align*}
Let us fix $\varepsilon = \frac{1}{(2M)^2}$, where $M$ is the constant obtained in the last inequality. It follows that there exists (another) constant $M > 0$ such that
\begin{equation} \label{eq:estimate_min_seq_3}
\| y_n \|_{W(0,T)} \leq M \| (y_0,f,g,h,q) \|_{\Upsilon_{T,0}}.
\end{equation}
Combined with \eqref{eq:estimate_min_seq_2}, we obtain that
\begin{equation*}
\| u_n \|_{L^2(0,T;U)} \leq M \| (y_0,f,g,h,q) \|_{\Upsilon_{T,0}}.
\end{equation*}
The sequence $(y_n,u_n)_{n \in \mathbb{N}}$ is therefore bounded in $W(0,T) \times L^2(0,T;U)$ and has a weak limit point $(y,u)$ 
satisfying
\begin{equation} \label{eq:estimate_yu}
\max \big( \| y \|_{W(0,T)}, \| u \|_{L^2(0,T;U)} \big)
\leq M \| (y_0,f,g,h,q) \|_{\Upsilon_{T,0}}.
\end{equation}
One can prove the optimality of $(y,u)$ with the same techniques as those used for the proof of \cite[Proposition 2]{BreKP17b}.

Consider now the solution $p$ to the adjoint system
\begin{equation} \label{eq:adjoint_lq}
-\dot{p} - A^* p - C^* C y = g, \quad p(T) - Q y(T)= q.
\end{equation}
The optimality conditions for the problem yield
$\alpha u + B^*p  + h= 0$,
see e.g.\@ \cite{HPUU09}. It follows that $(y,u,p)$ is a solution to \eqref{eq:non_reg_os}.

Let us prove the uniqueness. If $(y,u,p)$ is a solution to \eqref{eq:non_reg_os}, then one can prove that $(y,u)$ is a solution to problem \eqref{eq:lin_quad} with associated costate $p$. Therefore, it suffices to prove the uniqueness of the solution to \eqref{eq:non_reg_os}. To this end, it suffices to consider the case where $(y_0,f,g,h,q)=(0,0,0,0,0)$. Let $(y,u,p)$ be a solution to \eqref{eq:non_reg_os}. Then $(y,u)$ is a solution to \eqref{eq:lin_quad} and one can check that \eqref{eq:estimate_yu} holds. Thus, $(y,u)= (0,0)$ and then, $p= 0$, which proves the uniqueness.

It remains to prove the a priori bound. Observe that $(y,u,p)$ is the solution to 
\begin{equation} \label{eq:pi_sys}
\begin{cases}
\begin{array}{rll}
y(0) = & \! \! \! y_0 \qquad & \text{in $Y$} \\
\dot{y}-(Ay + Bu) = & \! \! \! f \qquad & \text{in $L^2(0,T;V^*)$} \\
-\dot{p} - A^* p - C^* C y = & \! \! \! g & \text{in $L^2(0,T;V^*)$} \\
\alpha u + B^*p = & \! \! \! - h & \text{in $L^2(0,T;U)$} \\
p(T) - \Pi y(T)= & \! \! \! \tilde{q} & \text{in $Y$},
\end{array}
\end{cases}
\end{equation}
where $\tilde{q}= (Q-\Pi) y(T) + q$. By \eqref{eq:estimate_yu}, we have $\| \tilde{q} \|_Y \leq M \big( \| (y_0,f,g,h,q) \|_{\Upsilon_{T,0}}$. Thus by Lemma \ref{lemma:linSysFiniteHorizon}, $\| (y,u,p) \|_{\Lambda_{T,0}} \leq M \| (y_0,f,g,h,\tilde{q}) \|_{\Upsilon_{T,0}}
\leq M \| (y_0,f,g,h,q) \|_{\Upsilon_{T,0}}$,
which concludes the proof.
\end{proof}

\subsubsection{Case of a negative weight}


\begin{proof}[Proof of Theorem \ref{theo:non_reg_os}: the case $\mu= -\lambda$]
Let $(y_0,f,g,h,q) \in \Upsilon_{T,-\lambda}$. The following inequality can be easily checked:
$\| (f,g,h) \|_{L^2(0,T)} \leq e^{\lambda T} \| (f,g,h) \|_{L_{-\lambda}^2(0,T)}$.
Therefore, by Lemma \ref{lemma:lq_pb}, the system \eqref{eq:non_reg_os} has a unique solution $(y,u,p)$, satisfying
\begin{align*}
\| (y,u,p) \|_{\Lambda_{T,0}}
\leq \ & M \max \big( \| y_0 \|_Y, \| (f,g,h) \|_{L^2(0,T)}, \| q \|_Y \big) \\
\leq \ & M \max \big( \| y_0 \|_Y, e^{\lambda T} \| (f,g,h) \|_{L_{-\lambda}^2(0,T)}, \| q \|_Y \big).
\end{align*}
It follows that $\| y(T) \|_Y \leq M \max \big( \| y_0 \|_Y, e^{\lambda T} \| (f,g,h) \|_{L_{-\lambda}^2(0,T)}, \| q \|_Y \big)$
and then that
\begin{align}
e^{-\lambda T} \| y(T) \|_Y \leq \ & M \max \big( e^{-\lambda T} \| y_0 \|_Y, \| (f,g,h) \|_{L_{-\lambda}^2(0,T)}, e^{-\lambda T} \| q \|_Y \big) \notag \\
\leq \ & M \| (y_0,f,g,h,q) \|_{\Upsilon_{T,-\lambda}}, \label{eq:estim_non_reg2}
\end{align}
since $e^{-\lambda T} \leq 1$.
The key idea now is to observe that $y(0)= y_0$, $\mathcal{H}(y,u,p)= (f,g,-h)$, and $p(T)-\Pi y(T)= \tilde{q}$, where $\tilde{q}= (Q-\Pi)y(T) + q$.
Thus, by Lemma \ref{lemma:linSysFiniteHorizon},
\begin{align}
\| (y,u,p) \|_{\Lambda_{T,-\lambda}}
\leq \ & M \| (y_0,f,g,h,\tilde{q}) \|_{\Upsilon_{T,-\lambda}} \notag \\
\leq \ & M \big( \| (y_0,f,g,h, q) \|_{\Upsilon_{T,-\lambda}} + e^{-\lambda T} \| (Q-\Pi) \|_{\mathcal{L}(Y)} \| y(T) \|_Y \big) \notag \\
\leq \ & M \| (y_0,f,g,h, q) \|_{\Upsilon_{T,-\lambda}} + Me^{-\lambda T} \| y(T) \|_Y, \label{eq:estim_non_reg3}
\end{align}
since $\mathcal{Q}$ is bounded.
Estimate \eqref{eq:estimate_non_reg_os} follows, combining \eqref{eq:estim_non_reg2} and \eqref{eq:estim_non_reg3}.
\end{proof}

\subsubsection{Case of positive weight}


The approach that we propose for dealing with the case $\mu= \lambda$ requires some more advanced tools, that we introduce now.
For a given $\theta \in (0,T)$, we make use of the following \emph{mixed} weighted space:
\begin{equation*}
\| u \|_{L^2_{\lambda,-\lambda}(0,T;U)}
= \| e^{\rho(\cdot)} u(\cdot) \|_{L^2(0,T;U)},
\end{equation*}
where
\begin{equation*}
\rho(t) = \lambda t, \text{ for $t \in [0,T-\theta]$,} \quad
\rho(t)= 2\lambda (T-\theta) - \lambda t, \text{ for $t \in [T-\theta,T]$}.
\end{equation*}
Observe that $\rho$ is continuous and piecewise affine, with $\dot{\rho}(t)= \lambda$ for $t \in [0,T-\theta)$ and $\dot{\rho}(t)= -\lambda$ for $t \in (T-\theta,T]$.
In a nutshell: We use a positive weight on $(0,T-\theta)$ and a negative weight on $(T-\theta,T)$.
We define similarly the space $L_{-\lambda,\lambda}^2(0,T;V^* \times V^* \times U)$ --- that we often denote by $L_{-\lambda,\lambda}^2(0,T)$ --- and the space $W_{\lambda,-\lambda}(0,T)$.
The spaces $\Lambda_{\lambda,-\lambda}$ and $\Upsilon_{\lambda,-\lambda}$ are defined in a similar way as before, with the corresponding norms
\begin{align*}
\| (y,u,p) \|_{\Lambda_{\lambda,-\lambda}}
= \ & \max \big( \| y \|_{W_{\lambda,-\lambda}(0,T)}, \| u \|_{L^2_{\lambda,-\lambda}(0,T;U)}, \| p \|_{W_{\lambda,-\lambda}(0,T)} \big), \\
\| (y_0,f,g,h,q) \|_{\Upsilon_{\lambda,-\lambda}}
= \ & \max \big( \| y_0 \|_Y, \| (f,g,h) \|_{L^2_{\lambda,-\lambda}(0,T)}, e^{\rho(T)} \| q \|_Y \big).
\end{align*}
The following lemma is a generalization of Lemma \ref{lemma:linSysFiniteHorizon} for mixed weighted spaces.

\begin{lemma} \label{lemma:linSysFiniteHorizon_mixte}
For all $T>0$, for all $(y_0,f,g,h,q) \in \Upsilon_{\lambda,-\lambda}$, the unique solution $(y,u,p)$ to \eqref{eq:non_reg_os} with $Q= \Pi$ satisfies the following bound:
\begin{equation} \label{eq:estimateForSysLin_mixte}
\| (y,u,p) \|_{\Lambda_{\lambda,-\lambda}}
\leq M \| (y_0,f,g,h,q) \|_{\Upsilon_{\lambda,-\lambda}},
\end{equation}
where $M$ is independent of $T$, $\theta$, and $(y_0,f,g,h,q)$.
\end{lemma}

\begin{proof}
We only give the main lines of the proof. One can obtain estimate \eqref{eq:estimateForSysLin_mixte} with the same decoupling as the one introduced in Lemma \ref{lemma:linSysFiniteHorizon}. The decoupled variables $y$ and $r$ can then be estimated in $W_{\lambda,-\lambda}(0,T)$, after an adaptation of Lemma \ref{lemma:exp_stab_forward_sys} for right-hand sides in $L_{\lambda,-\lambda}^2(0,T;V^*)$.
\end{proof}

\begin{proof}[Proof of Theorem \ref{theo:non_reg_os}: the case $\mu= \lambda$]
Let us first fix some constants. We denote by $M_1$ the constant involved in estimate \eqref{eq:estim_non_weight}. We denote by $M_2$ the constant involved in Lemma \ref{lemma:linSysFiniteHorizon_mixte}. Note that $M_1 \geq 1$ and $M_2 \geq 1$. Finally, $M_3$ denotes an upper bound on $\| Q- \Pi \|_{\mathcal{L}(Y)}$. Let us set
$M_0= 2 M_1 M_2 \geq 1$
and let us fix $\theta > 0$ such that
$M_0 M_3e^{-\lambda \theta} \leq 1$. 
The first four steps of this proof deal with the case where $T \geq \theta$. We will consider the case $T < \theta$ in Step 5.
Take now $T \geq \theta$ and $(y_0,f,g,h,q) \in \Upsilon_{T,\lambda}$. Since $\Upsilon_{T,\lambda}$ is embedded in $\Upsilon_{T,0}$, the existence of a solution to \eqref{eq:non_reg_os} in $\Lambda_{T,0}$ is guaranteed. Let us denote it by $(\bar{y},\bar{u},\bar{p})$.

\emph{Step 1:} construction of the mappings $\chi_1$ and $\chi_2$.\\
The main idea of the proof consists in obtaining an estimate of $\bar{y}(T)$ with a fixed-point argument. To this end, we introduce two affine mappings, $\chi_1$ and $\chi_2$, defined as follows: $\chi_1 \colon y_T \in Y \mapsto y(T-\theta) \in Y$,
where $y$ is the solution to
\begin{equation} \label{eq:def_chi_1}
\begin{cases}
\begin{array}{rll}
y(0) = & \! \! \! y_0 \qquad & \text{in $Y$} \\
\mathcal{H}(y,u,p)= & \! \! \! (f,g,-h) \quad & \text{in $L_{\lambda,-\lambda}^2(0,T;V^*\times V^* \times U)$} \\
p(T) - \Pi y(T)= & \! \! \! (Q-\Pi)y_T + q & \text{in $Y$}.
\end{array}
\end{cases}
\end{equation}
The mapping $\chi_2$ is defined as follows:
$\chi_2 \colon y_{T-\theta} \in Y \mapsto y(T) \in Y$,
where $y \in W(T-\theta,T)$ is the solution to
\begin{equation*}
\begin{cases}
\begin{array}{rll}
y(T-\theta) = & \! \! \! y_{T-\theta} \qquad & \text{in $Y$} \\
\mathcal{H}(y,u,p)= & \! \! \! (f,g,-h) & \text{in $L^2(T-\theta,T;V^*\times V^* \times U)$} \\
p(T) - Qy(T)= & \! \! \! q & \text{in $Y$}.
\end{array}
\end{cases}
\end{equation*}
The existence and uniqueness of a solution to the above system follows from Lemma \ref{lemma:lq_pb}, after a shifting of the time variable.
Observe that $\bar{y}(T-\theta)= \chi_1(\bar{y}(T))$ and that $\bar{y}(T)= \chi_2(\bar{y}(T-\theta))$. It follows that $\bar{y}(T)$ is a fixed point of $\chi_2 \circ \chi_1$.

\emph{Step 2:} on the Lipschitz-continuity of $\chi_1$ and $\chi_2$.\\
Let $y_T$ and $\tilde{y}_T \in Y$. We have $\chi_1(\tilde{y}_T)-\chi_1(y_T)= y(T-\theta)$, where $y$ is the solution to
\begin{equation*}
\begin{cases}
\begin{array}{rll}
y(0) = & \! \! \! 0 \qquad & \text{in $Y$} \\
\mathcal{H}(y,u,p)= & \! \! \! (0,0,0) & \text{in $L_{\lambda,-\lambda}^2(0,T;V^*\times V^* \times U)$} \\
p(T) - \Pi y(T)= & \! \! \! (Q-\Pi)(\tilde{y}_T-y_T) & \text{in $Y$}.
\end{array}
\end{cases}
\end{equation*}
By Lemma \ref{lemma:linSysFiniteHorizon_mixte},
\begin{equation*}
\| (y,u,p) \|_{\Lambda_{\lambda,-\lambda}}
\leq M_2 e^{\rho(T)} \| Q-\Pi \|_{\mathcal{L}(Y)} \| \tilde{y}_T-y_T \|_Y
\leq M_2 M_3 e^{\rho(T)} \| \tilde{y}_T-y_T \|_Y.
\end{equation*}
Thus, $e^{\rho(T-\theta)} \| y(T-\theta) \|_Y
\leq \| y \|_{W_{\lambda,-\lambda}(0,T)}
\leq M_2 M_3 e^{\rho(T)} \| \tilde{y}_T-y_T \|_Y$.
Observing that $e^{\rho(T)-\rho(T-\theta)} = e^{-\lambda \theta}$, we finally obtain that
\begin{equation*}
\| \chi_1(\tilde{y}_T) - \chi_1(y_T) \|_Y
= \| y(T-\theta) \|_Y
\leq M_2 M_3 e^{-\lambda \theta} \| \tilde{y}_T - y_T \|_Y,
\end{equation*}
which proves that $\chi_1$ is Lipschitz-continuous.
Now, let us take $y_{T-\theta}$ and $\tilde{y}_{T-\theta}$ in $Y$. We have $\chi_2(\tilde{y}_{T-\theta})- \chi_2(\tilde{y}_{T-\theta})= y(T)$, where $y \in W(T-\theta,T)$ is the solution to
\begin{equation*}
\begin{cases}
\begin{array}{rll}
y(T-\theta) = & \! \! \! \tilde{y}_{T-\theta}- y_{T-\theta} \qquad & \text{in $Y$} \\
\mathcal{H}(y,u,p)= & \! \! \! (0,0,0) & \text{in $L^2(T-\theta,T;V^*\times V^* \times U)$} \\
p(T) - Qy(T)= & \! \! \! 0 & \text{in $Y$}.
\end{array}
\end{cases}
\end{equation*}
We obtain with Lemma \ref{lemma:lq_pb} that
$\| y \|_{W(T-\theta,T)} \leq M_1 \| \tilde{y}_{T-\theta}- y_{T-\theta} \|_Y$
and thus
\begin{equation*}
\| \chi_2(\tilde{y}_{T-\theta})- \chi_1(y_{T-\theta}) \|_Y
= \| y(T) \| \leq M_1 \| \tilde{y}_{T-\theta}- y_{T-\theta} \|_Y,
\end{equation*}
proving that $\chi_2$ is Lipschitz-continuous.
As a consequence, the mapping $\chi_2 \circ \chi_1$ is Lipschitz-continuous, with modulus
$M_1 M_2 M_3 e^{-\lambda \theta}
\leq \frac{1}{2} M_0 M_3 e^{-\lambda \theta} \leq \frac{1}{2}$.

\emph{Step 3:} on the invariance of $B_Y \big(R \big)$, with $R= M_0e^{-\lambda(T-\theta)} \| (y_0,f,g,h,q) \|_{\Upsilon_{T,\lambda}}$.\\
Let $y_T \in B_Y(R)$.
Consider the solution $y$ to system \eqref{eq:def_chi_1}.
By Lemma \ref{lemma:linSysFiniteHorizon_mixte}, we have
\begin{equation} \label{eq:non_reg_analysis_2}
\| y \|_{W_{\lambda,-\lambda}(0,T)}
\leq M_2 \max \big( \| y_0 \|_Y, \| (f,g,h) \|_{L^2_{\lambda,-\lambda}(0,T)}, e^{\rho(T)} \| (Q-\Pi) y_{T} + q \|_Y \big).
\end{equation}
Let us estimate the last term in the above expression. We have
\begin{align}
& e^{\rho(T)} \| (Q-\Pi) y_T + q \|_Y \notag \\
& \qquad \quad \leq e^{\lambda T- 2 \lambda \theta} \big( M_3 \| y_T \|_Y  + \| q \|_Y \big) \notag \\
& \qquad \quad \leq e^{- \lambda \theta} M_0 M_3 \| (y_0,f,g,h,q) \|_{\Upsilon_{T,\lambda}} + e^{\lambda T} \| q \|_Y \notag \\
& \qquad \quad \leq \| (y_0,f,g,h,q) \|_{\Upsilon_{T,\lambda}} + e^{\lambda T} \| q \|_Y. \label{eq:non_reg_analysis_3}
\end{align}
Observe that $\| (f,g,h) \|_{L^2_{\lambda,-\lambda}(0,T)} \leq \| (f,g,h) \|_{L^2_\lambda(0,T)}$. Combining \eqref{eq:non_reg_analysis_2}, \eqref{eq:non_reg_analysis_3}, and this last observation, we obtain that
\begin{align*}
\| y \|_{W_{\lambda,-\lambda}(0,T)}
\leq \ & M_2 \max \big( \| y_0 \|_Y, \| (f,g,h) \|_{\lambda}, \| (y_0,f,g,h,q) \|_{\Upsilon_{T,\lambda}} + e^{\lambda T} \| q \|_Y \big) \\
\leq \ &  2 M_2 \| (y_0,f,g,h,q) \|_{\Upsilon_{T,\lambda}}.
\end{align*}
It follows then that
\begin{align}
\| \chi_1(y_T) \|_Y
= \ & e^{-\rho(T-\theta)} \| e^{\rho(T-\theta)} y(T-\theta) \|_Y \notag \\
\leq \ & e^{-\lambda(T-\theta)} \| y \|_{W_{\lambda,-\lambda}(0,T)} \notag \\
\leq \ & 2 M_2 e^{-\lambda(T-\theta)} \| (y_0,f,g,h,q) \|_{\Upsilon_{T,\lambda}}.
\label{eq:non_reg_analysis_11}
\end{align}
Applying now Lemma \ref{lemma:lq_pb}, we obtain that
\begin{align}
\| \chi_2 \circ \chi_1(y_T) \|_Y
\leq \ & M_1 \max \big( \| \chi_1(y_T) \|_Y, \| (f,g,h)_{|(T-\theta,T)} \|_{0}, \| q \|_Y \big). \label{eq:non_reg_analysis_1}
\end{align}
Observing that $e^{\lambda(T-\theta)} \| (f,g,h)_{|(T-\theta,T)} \|_{L^2(T-\theta,T)} \leq \| (f,g,h) \|_{L^2_\lambda(0,T)}$, we deduce from \eqref{eq:non_reg_analysis_11} and \eqref{eq:non_reg_analysis_1} that
\begin{align*}
\| \chi_2 \circ \chi_1(y_{T}) \|_Y
\leq \ & M_1 e^{-\lambda(T-\theta)} \max\big( 2 M_2  \|(y_0,f,g,h,q) \|_{\Upsilon_{T,\lambda}}, \| (f,g,h) \|_{\lambda}, e^{\lambda T} \| q \|_Y \big) \\
\leq \ & M_0 e^{-\lambda(T-\theta)} \| (y_0,f,g,h,q) \|_{\Upsilon_{T,\lambda}}.
\end{align*}
We have proved that $\| \chi_2 \circ \chi_1(y_T) \|_Y \leq R$.

\emph{Step 4:} proof of \eqref{eq:estimate_non_reg_os} (when $T \geq \theta$).\\
We have proved in the second step of the proof that $\chi_2 \circ \chi_1$ is a contraction. Therefore, $\bar{y}(T)$ is the unique fixed-point of $\chi_2 \circ \chi_1$ in $Y$. We have established in the third part of the proof that $B_Y (R)$ is invariant by $\chi_2 \circ \chi_1$. Therefore, by the fixed-point theorem, the mapping $\chi_2 \circ \chi_1$ has a unique fixed point in $B_Y(R)$ which is then necessarily $\bar{y}(T)$.

Observe now that $(\bar{y},\bar{u},\bar{p})$ is the solution to \eqref{eq:def_chi_1}, with $y_T= \bar{y}(T)$. Denoting by $M_4$ the constant involved in estimate \eqref{eq:estimateForSysLin}, we obtain that
\begin{align*}
\| (\bar{y},\bar{u},\bar{p}) \|_{\Lambda_{T,\lambda}}
\leq \ & M_4 \| (y_0,f,g,h, (Q-\Pi)\bar{y}(T) + q) \|_{\Upsilon_{T,\lambda}} \\
\leq \ & M_4 \big( \| (y_0,f,g,h,q) \|_{\Upsilon_{T,\lambda}} + M_3 e^{\lambda T} \| \bar{y}(T) \|_Y \big) \\
\leq \ & M_4 ( 1 + M_0 M_3 e^{\lambda \theta} ) \| (y_0,f,g,h,q) \|_{\Upsilon_{T,\lambda}}.
\end{align*}
This concludes the proof, in the case $T \geq \theta$.

\emph{Step 5:} proof of \eqref{eq:estimate_non_reg_os} (when $T < \theta$).\\
By Lemma \ref{lemma:embedding} and Lemma \ref{lemma:lq_pb}, we have
\begin{align*}
\| (y,u,p) \|_{\Lambda_{T,\lambda}}
\leq \ & M e^{\lambda T} \| (y,u,p) \|_{\Lambda_{T,0}} \\
\leq \ & M e^{\lambda \theta} \| (y_0,f,g,h,q) \|_{\Upsilon_{T,0}} \\
\leq \ & M \| (y_0,f,g,h,q) \|_{\Upsilon_{T,\lambda}},
\end{align*}
which proves \eqref{eq:estimate_non_reg_os} and concludes the proof of the theorem.
\end{proof}

\section{Additional results on optimality systems} \label{section:add_res}

In this subsection, we analyze further the optimality system associated with the linear-quadratic problem \eqref{eq:lin_quad} when $(f,g,h)=(0,0,0)$. Let us fix some notation. For $(y,u) \in W(0,T) \times L^2(0,T)$, we denote
\begin{equation} \label{eq:J_0}
J_{T,Q,q}^0(u,y)
= \int_0^T {\textstyle \frac{1}{2} } \| Cy(t) \|_{Z}^2  + {\textstyle \frac{\alpha}{2} } \|u(t)\|_U^2 \dd t + {\textstyle \frac{1}{2} } \langle y(T), Q y(T) \rangle_Y + \langle q, y(T) \rangle_Y
\end{equation}
and consider the problem
\begin{equation*}  \label{prob:P_0} \tag{$P^0$}
\mathcal{V}^0_{T,Q,q}(y_0)=
\begin{cases} \begin{array}{l}
{\displaystyle \inf_{ \begin{subarray}{c} y \in W(0,T) \\ u \in L^2(0,T;U) \end{subarray}} \ \
J_{T,Q,q}^0(u,y)} \\[1em]
\text{subject to: } \dot{y}= Ay + Bu, \quad y(0)= y_0.
\end{array}
\end{cases}
\end{equation*}
Problem \eqref{prob:P_0} is a particular case of problem \eqref{prob:turnpike} with $(f^{\diamond},g^{\diamond},h^{\diamond})= (0,0,0)$.
The associated optimality system is a linear system (of the form \eqref{eq:non_reg_os}) with parameters $(y_0,T,Q,q)$:
\begin{equation*} \label{eq:optim_sys_zero} \tag{$OS$}
y(0) = y_0, \quad \mathcal{H}(y,u,p)= (0,0,0), \quad p(T) - Q y(T)= q.
\end{equation*}
Since the solution $(y,u,p)$ is a linear mapping of $(y_0,q)$, there exist two linear operators $\Pi(T,Q)$ and $G(T,Q)$ such that
\begin{equation} \label{eq:def_pi_g}
p(0)= \Pi(T,Q) y_0 + G(T,Q)q.
\end{equation}
Let us mention that $\Pi(T,Q)$ can be described as the solution to a differential Riccati equation (see \cite[Part IV]{Benetal07}).

\begin{lemma} \label{lemma:decay_g}
There exists a constant $M>0$ such that for all $T>0$ and for all $Q \in \mathcal{Q}$, $\| \Pi(T,Q) \|_{\mathcal{L}(Y)} \leq M$, $\| G(T,Q) \|_{\mathcal{L}(Y)} \leq M e^{-\lambda T}$, and
\begin{align*}
\| \Pi(T,Q)- \Pi \|_{\mathcal{L}(Y)} \leq \ & M \| Q - \Pi \|_{\mathcal{L}(Y)} e^{-2 \lambda T}.
\end{align*} 
\end{lemma}

As a consequence of the last estimate, we obtain that $\Pi(T,Q) \underset{T \to \infty}{\longrightarrow} \Pi$ and that $\Pi(T,\Pi)= \Pi$. Let us mention that the third inequality has been obtained in \cite[Corollary 2.7]{PZ13} for finite-dimensional systems and that our result improves the one given in the same reference (see \cite[Lemma 3.9]{PZ13}), where a rate equal to $\lambda$ (instead of $2\lambda$) is established for parabolic systems.

\begin{proof}[Proof of Lemma \ref{lemma:decay_g}]
Applying Theorem \ref{theo:non_reg_os} with $\mu = -\lambda$, we obtain that
\begin{equation*}
\| e^{-\lambda \cdot} p(\cdot) \|_{L^\infty(0,T;Y)} \leq \| (y,u,p) \|_{\Lambda_{T,-\lambda}}
\leq M \max \big( \| y_0 \|_Y, e^{-\lambda T} \| q \|_Y \big)
\end{equation*}
and thus $\| p(0) \|_{Y} \leq M \max \big( \| y_0 \|_Y, e^{-\lambda T} \| q \|_Y \big)$. 
It follows that $\| \Pi(T,Q) \|_{\mathcal{L}(Y)} \leq M$ and that $\| G(T,Q) \|_{\mathcal{L}(Y)} \leq M e^{-\lambda T}$, as was to be proved.

Let us prove the last estimate. We take $q= 0$. Applying Theorem \ref{theo:non_reg_os} (with $\mu= \lambda$), we obtain that
$\| e^{\lambda \cdot} y(\cdot) \|_{L^\infty(0,T;Y)} \leq M \| y_0 \|_Y$.
Thus $\| y(T) \|_Y \leq Me^{-\lambda T} \| y_0 \|_Y$.
Let us set $r(t)= p(t)- \Pi y(t)$. We have $r(T)= (Q-\Pi)y(T)$, therefore
\begin{equation*}
\| r(T) \|_Y \leq M \| Q- \Pi \|_{\mathcal{L}(Y)} e^{-\lambda T} \| y_0 \|_Y.
\end{equation*}
Using the algebraic Riccati equation \eqref{eq:are} and the fact that $\Pi \in \mathcal{L}(V^*,Y) \cap \mathcal{L}(Y,V)$, one can check that $r \in W(0,T)$ and that $-\dot{r}= A_\pi^*r$. Since $A_\pi^* + \lambda I$ generates a bounded semigroup, we finally deduce that
\begin{equation*}
\| (\Pi(T,Q)-\Pi) y_0 \|_Y
= \| r(0) \|_Y
\leq M e^{-\lambda T} \| r(T) \|_Y
\leq M e^{-2\lambda T} \| Q - \Pi \|_{\mathcal{L}(Y)} \| y_0 \|_Y,
\end{equation*}
which concludes the proof.
\end{proof}

\begin{lemma} \label{lemma:diff_V}
Let $(\bar{y},\bar{u})$ be the solution to \eqref{prob:P_0} with associated costate $\bar{p}$. Let $(y,u) \in W(0,T) \times L^2(0,T;U)$ be such that $\dot{y}= Ay + Bu$. Then, there exists a constant $M$, independent of $T$, $Q$, $q$, $y_0$, $y$, and $u$ such that
\begin{align} 
0 \leq \ & J_{T,Q,q}^0(u,y)- \mathcal{V}_{T,Q,q}^0(y_0) - \langle \bar{p}(0), y(0)- y_0 \rangle_Y \notag \\
\leq \ & M \max \big( \| y - \bar{y} \|_{W(0,T)}^2, \| u - \bar{u} \|_{L^2(0,T;U)}^2 \big).
\label{eq:sensi}
\end{align}
\end{lemma}

\begin{proof}
We have
\begin{align}
& J_{T,Q,q}^0(u,y)- \mathcal{V}_{T,Q,q}^0(y_0)
 = J_{T,Q,q}^0(u,y) - J_{T,Q,q}^0(\bar{u},\bar{y}) \notag \\
& \qquad = \int_0^T \Big( {\textstyle \frac{1}{2} } \| C(y-\bar{y}) \|_Z^2 + {\textstyle \frac{\alpha}{2} } \|u-\bar{u}\|_U^2 +
 \langle C^*C \bar{y}, y-\bar{y} \rangle + \alpha \langle \bar{u},u-\bar{u} \rangle \Big) \dd t \notag \\
& \qquad \qquad + {\textstyle \frac{1}{2} } \langle y(T)-\bar{y}(T), Q(y(T)-\bar{y}(T)) \rangle_Y + \langle Q \bar{y}(T) + q, y(T)-\bar{y}(T) \rangle_Y. \label{eq:sensi_1}
\end{align}
The three quadratic terms can be bounded from above as follows:
\begin{align}
0 \leq & \ \int_0^T {\textstyle \frac{1}{2} } \| C(y-\bar{y}) \|_Z^2 + {\textstyle \frac{\alpha}{2} } \|u-\bar{u}\|_U^2 \dd t
+ {\textstyle \frac{1}{2} } \langle y(T)-\bar{y}(T), Q(y(T)-\bar{y}(T)) \rangle_Y \notag \\
\leq & \ M \max \big( \| y - \bar{y} \|_{W(0,T)}^2, \| u - \bar{u} \|_{L^2(0,T;U)}^2 \big).
\label{eq:sensi_2}
\end{align}
Let us focus on the remaining terms in the right-hand of \eqref{eq:sensi_1}. Using the relations $C^*C \bar{y} = -\dot{\bar{p}} - A^* \bar{p}$ and $\alpha \bar{u}= -B^* \bar{p}$ and integrating by parts, we obtain that
\begin{align}
& \int_0^T \langle C^*C \bar{y}, y-\bar{y} \rangle_Y + \alpha \langle \bar{u}, u-\bar{u}\rangle_U \dd t \notag \\
& \qquad = - \langle Q \bar{y}(T) + q, y(T)-\bar{y}(T) \rangle_Y
+ \langle \bar{p}(0), y(0)- y_0 \rangle_Y. \label{eq:sensi_3}
\end{align}
Estimate \eqref{eq:sensi} follows, by combining \eqref{eq:sensi_1}, \eqref{eq:sensi_2}, and \eqref{eq:sensi_3}.
\end{proof}

\begin{corollary} \label{coro:sensi_v_0}
The value function $\mathcal{V}_{T,Q,q}^0(\cdot)$ is differentiable. Moreover,
\begin{equation} \label{eq:sensi_v_0}
D_{y_{0}}  \mathcal{V}_{T,Q,q}^0 (y_0)
= \Pi(T,Q) y_0 + G(T,Q) q
\end{equation}
and $\Pi(T,Q)$ is self-adjoint and positive semi-definite.
\end{corollary}

\begin{proof}
Take $y_0 \in Y$ and $h \in Y$. Denote by $(\bar{y},\bar{u},\bar{p})$ and $(y,u,p)$ the solutions to \eqref{eq:optim_sys_zero} with initial conditions $y_0$ and $y_0 + h$, respectively. Then, by Theorem \ref{theo:non_reg_os},
\begin{equation*}
\max \big( \| y - \bar{y} \|_{W(0,T)}, \| u - \bar{u} \|_{L^2(0,T;U)} \big)
\leq \| (y,u,p) - (\bar{y},\bar{u},\bar{p}) \|_{\Upsilon_{T,0}}
\leq M \| h \|_Y.
\end{equation*}
Applying Lemma \ref{lemma:diff_V}, we deduce that
\begin{equation*}
0 \leq \mathcal{V}_{T,Q,q}^0(y_0 +h)- \mathcal{V}_{T,Q,q}^0(y_0) - \langle \bar{p}(0), h \rangle_Y \leq M \| h \|_Y^2,
\end{equation*}
which proves that $\mathcal{V}_{T,Q,q}^0$ is differentiable with $D_{y_0} \mathcal{V}_{T,Q,q}^0(y_0)= \bar{p}(0)$. Then \eqref{eq:sensi_v_0} follows with \eqref{eq:def_pi_g}.

Let us take now $q= 0$. Then, the solution $(y,u,p)$ to \eqref{eq:optim_sys_zero} is a linear mapping of $y_0$. Since $J_{T,Q,0}^0(u,y)$ is quadratic and convex, there exists a self-adjoint and positive semi-definite operator $\hat{\Pi}(T)$ such that $\mathcal{V}_{T,Q,0}^0(y_0)= \frac{1}{2} \langle y_0, \hat{\Pi}(T) y_0 \rangle$. Applying the first part of the lemma, we deduce that for all $y_0 \in Y$,
$D_{y_0} \mathcal{V}_{T,Q,0}^0 (y_0) = \hat{\Pi}(T) y_0 = \Pi(T,Q) y_0$,
which proves that $\hat{\Pi}(T)= \Pi(T,Q)$ and concludes the proof.
\end{proof}

\section{Linear-quadratic problems} \label{section:lq_pb}

\subsection{Turnpike property}

We analyze now the class of problems \eqref{prob:turnpike} (defined in the introduction).
By Lemma \ref{lemma:lq_pb}, \eqref{prob:turnpike} has a unique solution $(\bar{y},\bar{u})$ with associated costate $\bar{p}$, satisfying
\begin{equation} \label{eq:oc_for_lq_pb}
\bar{y}(0) = y_0, \quad
\mathcal{H}(y,u,p)= (f^\diamond,g^\diamond,-h^\diamond), \quad
\bar{p}(\bar{T}) - Q \bar{y}(\bar{T})= q.
\end{equation}
Note that the variables $f^\diamond$, $g^\diamond$, and $h^\diamond$ must be understood as constant time-functions in the above optimality system.
Let us first investigate the existence of a solution to the static optimization problem.

\begin{lemma} \label{lemma:existence_turnpike}
The static optimization problem \eqref{eq:turnpike_pb_setub}
has a unique solution $(y^\diamond,u^\diamond)$ with unique associated Lagrange multiplier $p^\diamond \in V$, i.e.\@ $p^\diamond$ is such that
\begin{equation} \label{eq:turnpike_oc}
-(Ay^\diamond + Bu^\diamond) = f^\diamond, \quad
- A^* p^\diamond - C^* C y^\diamond = g^\diamond, \quad
\alpha u^\diamond + B^* p^\diamond  = - h^\diamond.
\end{equation}
Moreover, there exists a constant $M > 0$, independent of $(f^\diamond,g^\diamond,h^\diamond )$, such that
\begin{equation} \label{eq:static_prob_turnpike}
\max \big( \| y^\diamond \|_V, \|u^\diamond\|_U,  \| p^\diamond \|_V \big) \leq M \max \big( \| f^\diamond \|_{V^*}, \| g^\diamond \|_{V^*}, \|h^\diamond \|_U \big).
\end{equation}
\end{lemma}

\begin{proof}
Since by \cite[page 207, equation 2.7]{Benetal07} (with $\alpha = \frac{1}{2}$) the operator $A_{\pi}$ is an isomorphism from $V$ to $V^*$, we can define $r^\diamond = -A_\pi^{-*}(\Pi f^\diamond -\frac{1}{\alpha} \Pi Bh^\diamond + g^\diamond) \in V$.
Similarly to the proof of Lemma \ref{lemma:linSysFiniteHorizon} we next define $y^{\diamond} = A_\pi^{-1} \big( \frac{1}{\alpha} Bh^{\diamond} - f^\diamond +\frac{1}{\alpha} BB^* r^{\diamond} \big) \in  V$, $p^{\diamond} = \Pi y^\diamond + r^\diamond \in V$, and $u^{\diamond} = -\frac{1}{\alpha} (h^\diamond + B^* p^\diamond) \in U$.
It is easily verified that the triplet $(y^{\diamond},p^{\diamond},u^{\diamond})$ is a solution to  \eqref{eq:turnpike_oc} and that it satisfies  \eqref{eq:static_prob_turnpike}.

It remains to discuss the uniqueness of the solution to \eqref{eq:turnpike_pb_setub} and the uniqueness of the solution to \eqref{eq:turnpike_oc}. Let us first remark that if $(y,u,p)$ is solution to \eqref{eq:turnpike_oc}, then $(y,u)$ is solution to \eqref{eq:turnpike_pb_setub} with associated Lagrange multiplier $p$, by convexity of the optimization problem. Therefore, the uniqueness of the solution to \eqref{eq:turnpike_oc} implies the uniqueness of the solution to \eqref{eq:turnpike_pb_setub}.

To prove the uniqueness of the solution to \eqref{eq:turnpike_oc}, it suffices to consider the case $(f^\diamond,g^\diamond,h^\diamond)= (0,0,0)$. Let $(y,u,p)$ be a solution to \eqref{eq:turnpike_oc} with $(f^\diamond,g^\diamond,h^\diamond)=(0,0,0)$. Let us define $r=p-\Pi y$. It then follows that $A_\pi^*r=0$ and, hence, $r=0$. Consequently, we have $\Pi y=p$ and with $Ay=-Bu$ we conclude that $A_\pi y =0$. This implies $y=0$ and $p=\Pi y =0$. Since $\alpha u+ B^*p =0,$ we finally obtain that $u = 0$, which concludes the proof the lemma.
\end{proof}

From now on, we denote
\begin{equation} \label{eq:def_tilde_q}
\tilde{q}= q - p^\diamond + Qy^\diamond.
\end{equation}
We state and prove in Theorem \ref{theorem:turnpike} the turnpike property announced in the introduction. A consequence of inequality \eqref{eq:turnpike_estimate} below is that if $t$ is not too close to $0$ and not too close to $\bar{T}$, then $\bar{y}(t)$ and $\bar{p}(t)$ are close to $y^\diamond$ and $p^\diamond$, respectively.

\begin{theorem} \label{theorem:turnpike}
There exists a constant $M$, independent of the parameters $\bar{T}$, $Q$, and $(y_0,f^\diamond,g^\diamond,h^\diamond,q)$ such that for all $t \in [0,\bar{T}]$,
\begin{equation}
\max \big( \| \bar{y}(t) - y^\diamond \|_Y, \| \bar{p}(t) - p^\diamond \|_Y \big) \leq M \big( e^{-\lambda t} \| y_0 - y^\diamond \|_Y + e^{-\lambda(\bar{T}-t)} \| \tilde{q} \|_Y \big). \label{eq:turnpike_estimate}
\end{equation}
\end{theorem}

\begin{remark}
The exponential turnpike property established in \cite{PZ13,TZ15} takes the following form: $\max \big( \| \bar{y}(t) - y^\diamond \|_Y, \| \bar{p}(t) - p^\diamond \|_Y \big) \leq M_1  e^{-\lambda t} + M_2 e^{-\lambda(\bar{T}-t)}$, where the constants $M_1$ and $M_2$ depend on all the data of the problem (except $\bar{T}$). Our estimate is thus more precise: It shows that these two constants are related to $\| y_0 - y^\diamond \|_Y$ and $\| q - p^\diamond + Qy^\diamond \|_Y$, respectively. 
\end{remark}

\begin{proof}[Proof of Theorem \ref{theorem:turnpike}]
Let $(\tilde{y},\tilde{u},\tilde{p})= (\bar{y},\bar{u},\bar{p})-(y^\diamond,u^\diamond,p^{\diamond})$.
We have
\begin{equation*}
\tilde{p}(\bar{T})-Q \tilde{y}(\bar{T})
= p(\bar{T}) - p^\diamond - Q(y(\bar{T})- y^\diamond)
= q + Q y^\diamond - p^\diamond = \tilde{q}.
\end{equation*}
Then, by \eqref{eq:oc_for_lq_pb} and \eqref{eq:turnpike_oc}, $\tilde{y}(0) = y_0 - y^\diamond$, $\mathcal{H}(y,u,p)=(0,0,0)$, $\tilde{p}(\bar{T}) - Q \tilde{y}(\bar{T})= \tilde{q}$,
i.e.\@ $(\tilde{y},\tilde{u},\tilde{p})$ is the solution to \eqref{eq:optim_sys_zero}, with parameters $(y_0 - y^\diamond, \bar{T},Q, \tilde{q})$.
Let $(y^{(1)},u^{(1)},p^{(1)})$ and $(y^{(2)},u^{(2)},p^{(2)})$ be the solutions to \eqref{eq:optim_sys_zero}, with parameters $(y_0 - y^\diamond, \bar{T},Q, 0)$ and $(0,\bar{T},Q, \tilde{q})$ respectively.
Applying Theorem \ref{theo:non_reg_os} to these systems with $\mu= \lambda$ and $\mu= -\lambda$ respectively, we obtain that
\begin{align*}
\| (y^{(1)},u^{(1)},p^{(1)}) \|_{\Lambda_{T,\lambda}} \leq \ & M \| y_0-y^\diamond \|_Y, \\
\| (y^{(2)},u^{(2)},p^{(2)}) \|_{\Lambda_{T,-\lambda}} \leq \ & M e^{-\lambda \bar{T}} \| \tilde{q} \|_Y
\end{align*}
We immediately deduce that for all $t \in [0,\bar{T}]$
\begin{align*}
\max \big( \| y^{(1)}(t) \|_Y, \| p^{(1)}(t) \|_Y \big)
\leq \ & M e^{-\lambda t} \| y_0-y^\diamond \|_Y, \\
\max \big( \| y^{(2)}(t) \|_Y, \| p^{(2)}(t) \|_Y \big)
\leq \ & M e^{-\lambda (\bar{T} - t)} \| \tilde{q} \|_Y.
\end{align*}
Estimate \ref{eq:turnpike_estimate} follows, since by linearity, $(\tilde{y},\tilde{u},\tilde{p})= (y^{(1)},u^{(1)},p^{(1)}) + (y^{(2)},u^{(2)},p^{(2)})$.
\end{proof}

\begin{remark}
If one assumes that $B \in \mathcal{L}(U,Y)$ (instead of simply $B \in \mathcal{L}(U, V^*)$), then a turnpike property can also be established for the control:
\begin{align*}
\|u(t)- u^\diamond\|_U
= {\textstyle \frac{1}{\alpha} } \big\| B^*(p(t)-p^\diamond) \big\|_U 
\leq M \big( e^{-\lambda t} \| y_0 - y^\diamond \|_Y + e^{-\lambda(\bar{T}-t)} \| \tilde{q} \|_Y \big).
\end{align*}
\end{remark}

\subsection{Analysis of the value function}

In this subsection, we analyze some properties of the value function associated with Problem \eqref{prob:turnpike}. For an initial time $\theta$ and an initial condition $y_\theta$, the value function is defined by
\begin{equation*} \label{prob:dyn_prog} \tag{$P(\theta)$}
\mathcal{V}_{\bar{T},Q,q}(\theta,y_\theta)=
\left\{
\begin{array}{l}
{\displaystyle
\inf_{\begin{subarray}{c} y \in W(\theta,\bar{T}) \\ u \in L^2(\theta,\bar{T};U) \end{subarray}}
 \int_{\theta}^{\bar{T}} \ell(y(t),u(t)) \dd t + {\textstyle \frac{1}{2} } \langle y(\bar{T}), Q y(\bar{T}) \rangle_Y + \langle q, y(\bar{T}) \rangle_Y, } \\[1em]
\text{subject to: } \dot{y}(t)= Ay(t) + Bu(t) + f^\diamond, \quad y(\theta)= y_\theta.
\end{array}
\right.
\end{equation*}
The shifting realized in the proof of Theorem \ref{theorem:turnpike} shows that Problem \eqref{prob:turnpike} is equivalent to a problem of the same form as \eqref{prob:P_0} (with a different value of $q$). We compare the corresponding value functions in the next lemma.

\begin{lemma} \label{lemma:link_value_function}
The following relation holds true:
\begin{align}
\mathcal{V}_{\bar{T},Q,q}(\theta,y_\theta)
= \ & \mathcal{V}^0_{\bar{T}-\theta,Q,\tilde{q}}(y_\theta - y^\diamond)
+ \langle p^\diamond, y_\theta \rangle_Y \notag \\
& \qquad + (\bar{T} - \theta) v^\diamond
+ {\textstyle \frac{1}{2} } \langle y^\diamond, Q y^\diamond \rangle_Y
+ \langle q- p^\diamond, y^\diamond \rangle,
\label{eq:link_V_V}
\end{align}
where $v^\diamond := \ell(y^\diamond,u^\diamond)$ is the value of the static optimization problem \eqref{eq:turnpike_pb_setub}.
\end{lemma}

\begin{proof}
It is sufficient to prove the result for $\theta= 0$.
Let $(y,u)$ be such that $\dot{y}= Ay + Bu + f^\diamond$, $y(0)= y_0$.
Let $(\tilde{y},\tilde{u})= (y,u)-(y^\diamond,u^\diamond)$. Then, $\dot{\tilde{y}}= A \tilde{y} + B \tilde{u}$, $\tilde{y}(0)= y_0 - y^\diamond$.
We have
\begin{align}
J_{\bar{T},Q,q}(u,y)= & \ \int_{0}^{\bar{T}}
{\textstyle \frac{1}{2} } \| C \tilde{y} \|_Z^2 + \langle C^*C y^\diamond + g^\diamond, \tilde{y} \rangle_{V^*,V} + \Big( {\textstyle \frac{1}{2} } \| C y^\diamond \|_Z^2 + \langle g^\diamond, y^\diamond \rangle_Y \Big) \dd t \notag \\
& \qquad + \int_{0}^{\bar{T}} {\textstyle \frac{\alpha}{2} } \| \tilde{u}(t) \|_U^2 + \langle \alpha u^\diamond + h^\diamond, \tilde{u}(t)\rangle_U + \Big( {\textstyle \frac{\alpha}{2} } \|u^\diamond\|_U^2 + \langle h^\diamond, u^\diamond \rangle_U \Big) \dd t \notag \\
& \qquad + {\textstyle \frac{1}{2} } \langle \tilde{y}(T), Q \tilde{y}(T) \rangle_Y
+ \langle Q y^\diamond + q, \tilde{y}(T) \rangle_Y
+ {\textstyle \frac{1}{2} } \langle y^\diamond, Q y^\diamond \rangle_Y + \langle q, y^\diamond \rangle_Y.
\label{eq:equi_turn_1}
\end{align}
As in the proof of Lemma \ref{lemma:diff_V}, the linear terms vanish. Using $C^*C y^\diamond + g^\diamond = - A^* p^\diamond$, $\alpha u^\diamond + h^\diamond = -B^* p^\diamond$, and integrating by parts, one indeed obtains that
\begin{equation}
\int_{0}^{\bar{T}}
\langle C^*C y^\diamond + g^\diamond, \tilde{y}(t) \rangle_{V^*,V}
+
\langle \alpha u^\diamond + h^\diamond, \tilde{u}(t)\rangle_U \dd t
=
-\langle p^\diamond, \tilde{y}(\bar{T})- \tilde{y}(0) \rangle_Y. \label{eq:equi_turn_2}
\end{equation}
Combining \eqref{eq:equi_turn_1} and \eqref{eq:equi_turn_2}, we obtain that
\begin{align*}
J_{\bar{T},Q,q}(u,y)= & \ \int_{0}^{\bar{T}}
{\textstyle \frac{1}{2} } \| C \tilde{y} \|_Z^2 + {\textstyle \frac{\alpha}{2} } \| \tilde{u}(t)\|_U^2  \dd t \\
& \qquad + {\textstyle \frac{1}{2} } \langle \tilde{y}(\bar{T}), Q \tilde{y}(\bar{T}) \rangle_Y
+ \langle Q y^\diamond + q - p^\diamond, \tilde{y}(\bar{T}) \rangle_Y + K(y_0),
\end{align*}
where $K(y_0) = \bar{T} v^\diamond + {\textstyle \frac{1}{2} } \langle y^\diamond, Q y^\diamond \rangle_Y
+ \langle p^\diamond, y_0 - y^\diamond \rangle_Y
+ \langle q, y^\diamond \rangle_Y$.
We obtain with the definitions of $J_0$ and $\tilde{q}$ given in \eqref{eq:J_0} and \eqref{eq:def_tilde_q} that
$J_{\bar{T},Q,q}(u,y)= J_{\bar{T},Q,\tilde{q}}^0(\tilde{u},\tilde{y}) + K(y_0)$.
Therefore
$\mathcal{V}_{\bar{T},Q,q}(0,y_0)
= \mathcal{V}_{\bar{T}, Q, \tilde{q}}^0(y_0-y^\diamond) + K(y_0)$
and the lemma is proved.
\end{proof}

We deduce from Lemma \ref{lemma:link_value_function} some useful information on $D_{y_{\theta}} \mathcal{V}_{\bar{T},Q,q}(\theta,y_{\theta})$. More precisely, relation \eqref{coro:sensi_for_real_pb} below shows how the derivative of the value function deviates from the equilibrium value $p^\diamond$. Note that the first difference term, $\Pi(\bar{T} - \theta,Q) (y_{\theta}-y^\diamond)$, vanishes when $y_{\theta}= y^{\diamond}$ and the second one, $G(\bar{T} - \theta,Q) \tilde{q}$, is very small for large values of $\bar{T}-\theta$.

\begin{corollary} \label{coro:sensi_for_real_pb}
The following relation holds true:
\begin{equation} \label{eq:sensi_01}
D_{y_{\theta}} \mathcal{V}_{\bar{T},Q,q}(\theta,y_{\theta})
= \Pi(\bar{T} - \theta,Q) (y_{\theta}-y^\diamond) + G(\bar{T} - \theta,Q) \tilde{q} + p^\diamond.
\end{equation}
Moreover, for all $\theta \in [0,\bar{T}]$,
\begin{equation} \label{eq:sensi_02}
\bar{p}(\theta)= \Pi(\bar{T} - \theta,Q) (\bar{y}(\theta)-y^\diamond) + G(\bar{T} - \theta,Q) \tilde{q} + p^\diamond.
\end{equation}
\end{corollary}

\begin{proof}
Relation \eqref{eq:sensi_01} is obtained by differentiating relation \eqref{eq:link_V_V} and applying Corollary \ref{coro:sensi_v_0}. Using the same techniques as in Lemma \ref{lemma:diff_V} and Corollary \ref{coro:sensi_v_0}, one can prove the following sensitivity relation: $\bar{p}(\theta)=
D_{y_{\theta}} \mathcal{V}_{\bar{T},Q,q}(\theta,\bar{y}(\theta))$.
Applying \eqref{eq:sensi_01}, relation \eqref{eq:sensi_02} follows.
\end{proof}

\section{Error estimate for the RHC algorithm} \label{section:rhc}

The receding-horizon algorithm for solving \eqref{prob:turnpike} consists in solving a sequence of optimal control problems with small time-horizon $T$. A sampling time $\tau \leq T$ is fixed. At iteration $n$ of the algorithm, an optimal control problem is solved on the interval $(n\tau,n\tau + T)$ and only the restriction to $(n\tau,(n+1) \tau)$ of the solution is kept. 
The problem which is solved at the iteration $n$ is of the following form:
\begin{equation} \tag{$P(\theta;\phi)$} \label{prob:inter}
\begin{cases}
\begin{array}{l}
{\displaystyle
\inf_{\begin{subarray}{c} y \in W(\theta,\theta + T) \\ u \in L^2(\theta,\theta + T;U) \end{subarray}} \
\int_{\theta}^{\theta + T} \ell(y(t),u(t)) \dd t  + \phi(\theta + T,y(\theta + T)), } \\[2em]
\text{subject to: } \dot{y}(t)= Ay(t) + Bu(t) + f^\diamond, \quad y(\theta)= y_{\theta},
\end{array}
\end{cases}
\end{equation}
where $\theta$ and $y_\theta$ are given.
Let us describe the function $\phi$ used as final-time cost in the above problem.
We assume that two bounded mappings $\tilde{\Pi} \colon t \in [0,\infty) \rightarrow \mathcal{L}(Y)$ and $\tilde{G} \colon t \in [0,\infty) \rightarrow \mathcal{L}(Y)$ are given as well as an element $\tilde{p} \in Y$. For all $t \geq 0$, the operator $\tilde{\Pi}(t)$ is assumed to be self-adjoint and positive semi-definite.
The function $\phi$ is defined by
\begin{equation} \label{eq:def_phi_1}
\phi(t,y)= {\textstyle \frac{1}{2} } \langle y-y^\diamond, \tilde{\Pi}(\bar{T}-t) (y-y^\diamond) \rangle_Y
+ \langle \tilde{G}(\bar{T}-t) \tilde{q}, y \rangle
+ \langle \tilde{p}, y \rangle_Y. 
\end{equation}
Observe that
\begin{equation*}
D_y \phi(\theta + T,y)= 
\tilde{\Pi} \big( \bar{T}-(\theta + T) \big) (y-y^\diamond)
+ \tilde{G} \big( \bar{T}-(\theta + T) \big) \tilde{q} + \tilde{p}.
\end{equation*}
This relation shows that $\phi(\theta + T,\cdot)$ can be viewed as an approximation of the value function $\mathcal{V}_{\bar{T},Q,q}(\theta + T, \cdot)$ (up to an additive constant independent of the variable $y$). If $\tilde{p}= p^\diamond$ and if $\tilde{\Pi}$ and $\Pi(\cdot, Q)$ as well as $\tilde{G}$ and $G(\cdot, Q)$ coincide at time $\bar{T}-(\theta + T)$, then the two problems \eqref{prob:dyn_prog} and \eqref{prob:inter} are equivalent, by the dynamic programming principle.

A third parameter $N$ such that $N \tau \leq \bar{T}$ is also considered. At time $N \tau$, Problem \eqref{prob:dyn_prog} is solved (with $\theta= N \tau$).
We give now a precise description of the algorithm.

\begin{algorithm}[H]
\begin{algorithmic}
\STATE Input: $\tau \geq 0$, $T \geq \tau$, and $N$ such that $N \tau \leq \bar{T}$;
\FOR{$n=0,1,2,...,N-1$}
\STATE Find the solution $(y,u)$ to \eqref{prob:inter} with $\theta= n\tau$, $y_\theta= y_n$, and $\phi$ given by \eqref{eq:def_phi_1};
\STATE Set $y_{RH}(t)= y(t)$ and $u_{RH}(t)= u(t)$ for a.e.\@ $t \in (n \tau, (n+1) \tau)$;
\STATE Set $y_{n+1}= y(\tau)$;
\ENDFOR
\STATE Find the solution $(y,u)$ to Problem \eqref{prob:dyn_prog} with $\theta= N \tau$ and $y_\theta= y_{N}$;
\STATE Set $y_{RH}(t)= y(t)$ and $u_{RH}(t)= u(t)$ for a.e.\@ $t \in (N \tau,\bar{T})$;
\end{algorithmic}
\caption{Receding-Horizon method}
\label{algo:rh_turnpike}
\end{algorithm}

We are now ready to state and prove the main result of the article. We make use of the following assumptions on $\tilde{\Pi}$ and $\tilde{G}$.

\begin{hypothesis} \label{assumption:pi_g}
For all $t \geq 0$, $\tilde{\Pi}(t)$ is self-adjoint positive semi-definite. There exists a constant $M > 0$ such that $\| \tilde{G}(t) \|_{\mathcal{L}(Y)} \leq M e^{-\lambda t}$ and $\| \tilde{\Pi}(t) \|_{\mathcal{L}(Y)} \leq M$, $\forall t \geq 0$.
\end{hypothesis}

Let us remark that a simple possible choice is $\tilde{\Pi}= 0$, $\tilde{G}= 0$. In this situation, we then have $\phi(t,y)= \langle \tilde{p}, y \rangle_Y$. We denote
\begin{align*}
\| \tilde{\Pi}- \Pi \|_{\infty} = \ & \sup_{T \in [0,\infty)} \| \tilde{\Pi}(T) - \Pi(T,Q) \|_{\mathcal{L}(Y)} \\
\| \tilde{G}- G \|_{\infty,\lambda} = \ & \sup_{T \in [0,\infty)} \| e^{\lambda T} (\tilde{G}(T)-G(T,Q)) \|_{\mathcal{L}(Y)}.
\end{align*}
By Assumption \ref{assumption:pi_g} and Lemma \ref{lemma:decay_g}, $\| \tilde{\Pi} - \Pi \|_\infty$ and $\| \tilde{G} - G \|_{\infty,\lambda}$ are finite.

\begin{theorem} \label{theo:main}
There exist two constants $\tau_0>0$ and $M>0$ such that for all $\tau$ and $T$ with $\tau_0 \leq \tau \leq T \leq \bar{T}$ and for all $N$ with $N\tau \leq T$, the following estimate holds true:
\begin{align}
& \max \big( \| y_{RH}- \bar{y} \|_{W(0,\bar{T})}, \| u_{RH} - \bar{u} \|_{L^2(0,\bar{T};U)} \big) \notag \\
& \qquad \qquad \leq M e^{-\lambda(T-\tau)} \big( e^{-\lambda T} K_1 + e^{-\lambda(\bar{T}- (N \tau + T))} K_2 + N \| \tilde{p} - p^\diamond \|_Y \big), \label{eq:theo_main_1}
\end{align}
where $K_1= \| \tilde{\Pi}- \Pi \|_\infty \| y_0 - y^\diamond \|_Y$ and $K_2= \big( \| \tilde{\Pi}- \Pi \|_\infty + \| \tilde{G} - G \|_{\infty,\lambda} \big) \| \tilde{q} \|_Y$.
Moreover,
\begin{align}
& J_{\bar{T},Q,q}(y_{RH},u_{RH}) - \mathcal{V}_{\bar{T},Q,q}(0,y_0) \notag \\
& \qquad \leq
M e^{-2\lambda(T-\tau)} \big( e^{-\lambda T} K_1 + e^{-\lambda(\bar{T}- (N \tau + T))} K_2 + N \| \tilde{p} - p^\diamond \|_Y \big)^2.
\label{eq:theo_main_2}
\end{align}
The constant $M$ is independent of $(y_0,f^\diamond,g^\diamond,h^\diamond,q)$, $Q$, $\bar{T}$, $\tau$, $T$, and $N$.
\end{theorem}

\begin{remark}
Estimate \eqref{eq:theo_main_1} suggests that the quality of the solution provided by the Receding-Horizon algorithm can be improved by either reducing $\tau$, by increasing $T$, or by reducing $N$, which is intuitive. Let us mention, however, that the constant $\tau_0$ constructed in the proof cannot be chosen arbitrarily small, therefore, our result does not give information on the quality of the solution for arbitrarily small sampling times.

The error estimate also suggests to choose $\tilde{p}= p^\diamond$. In this case, one can recommend to choose $N$ such that $N \approx (\bar{T} - 2T)/\tau$, so that the two error terms $e^{-\lambda T} K_1$ and $e^{-\lambda(\bar{T}- (N \tau + T))} K_2$ are of the same order (with respect to $T$).
\end{remark}

\begin{remark}
The necessity of a lower bound $\tau_0$ for the sampling time is revealed in the proof below; in a nutshell, this lower bound ultimately allows to sum up the error terms accumulated at each iteration of the algorithm. Let us mention that this bound is not necessary in other works based on a dynamic programming approach and dealing with continuous-time systems. Still in those works, a lower bound on the prediction horizon $T$, depending on $\tau$, is needed (see \cite{AK2,AK1,AK3}).
\end{remark}

\begin{proof}[Proof of Theorem \ref{theo:main}]
Let us set define, for $n \in \{ 0,...,N-1 \}$,
\begin{align*}
a_n = \ & \max \big( \| y_{RH}- \bar{y} \|_{W(n \tau, (n+1) \tau )}, \| u_{RH} - \bar{u} \|_{L^2(n \tau, (n+1) \tau;U)} \big) \\
b_n = \ & \| y_{RH}(n\tau) - \bar{y}(n\tau) \|_Y.
\end{align*}
We also define $a_{N}= \max \big( \| y_{RH}- \bar{y} \|_{W(N \tau, \bar{T})}, \| u_{RH} - \bar{u} \|_{L^2(N \tau, \bar{T};U)} \big)$.
Let $M_0$ be the constant involved in Theorem \ref{theo:non_reg_os}, for $\mu= \lambda$ and for $\mathcal{Q}= \{ \tilde{\Pi}(t) \,|\, t \geq 0 \}$. Necessarily, $M_0 \geq 1$. Let $r \in (0,1)$ be a fixed real number and let the constant $\tau_0>0$ be such that $e^{-\lambda \tau_0} \leq M_0 e^{-\lambda \tau_0} < r < 1$.

\emph{Step 1: }proof of estimates on $a_n$ and $b_n$. \\
The first part of the proof consists in proving the following three estimates.
\begin{align}
a_n \leq \ & M b_n
+ M e^{-\lambda(T-\tau)} \big( e^{-\lambda(n\tau + T)} K_1
+ e^{-\lambda(\bar{T}-(n \tau + T))} K_2 + \| \tilde{p} - p^\diamond \|_Y \big), \label{eq:estimate_an} \\
b_{n+1} \leq \ & r b_n
+ M e^{-\lambda(T-\tau)} \big( e^{-\lambda(n\tau + T)} K_1
+ e^{-\lambda(\bar{T}-(n \tau + T))} K_2 + \| \tilde{p} - p^\diamond \|_Y \big), \label{eq:estimate_bn} \\
a_{N} \leq \ & M b_{N}, \label{eq:estimate_an_max}
\end{align}
for all $n = 0,...,N-1$.
Let us set $t_n= n \tau$ and $t_n'= n\tau + T$, for all $n=0,...,N$.
We also set $\bar{y}_n= \bar{y}(n \tau)$ and recall that $y_n= y_{RH}(n\tau)$.
Let us denote by $(y,u)$ the solution to problem \eqref{prob:inter} with $\theta= n \tau$ and $y_\theta= y_n$. Let $p$ be the associated costate. By construction, $(y_{RH},u_{RH})$ and $(y,u)$ coincide on the interval $(t_n,t_{n+1})$.
Let us write the optimality conditions satisfied by $(\bar{y},\bar{u},\bar{p})$ and $(y,u,p)$ on the interval $(t_n,t_n')$. By Corollary \ref{coro:sensi_for_real_pb}, we have
\begin{equation*}
\begin{cases}
\begin{array}{rl}
\bar{y}(t_n) = & \! \! \! \bar{y}_n  \\
\mathcal{H}(\bar{y},\bar{u},\bar{p})= & \! \! \! (f^\diamond,g^\diamond,-h^\diamond) \\
\bar{p}(t_n') - \Pi(\bar{T}-t_n',Q) (\bar{y}(t_n') - y^\diamond)= & \! \! \! G(t_n') \tilde{q} + p^\diamond.
\end{array}
\end{cases}
\end{equation*}
The optimality conditions associated with $(y,u,p)$ write
\begin{equation*}
\begin{cases}
\begin{array}{rl}
y(t_n) = & \! \! \! y_n  \\
\mathcal{H}(y,u,p)= & \! \! \! (f^\diamond,g^\diamond,-h^\diamond) \\
p(t_n') - \tilde{\Pi}(\bar{T}-t_n') (y(t_n')- y^\diamond)= & \! \! \! \tilde{G}(t_n') \tilde{q} + \tilde{p}.
\end{array}
\end{cases}
\end{equation*}
Thus, the triple $(\hat{y},\hat{u},\hat{p})(t):=(y,u,p)(t_n + t)-(\bar{y},\bar{u},\bar{p})(t_n + t)$ satisfies
\begin{equation} \label{eq:sys_for_diff}
\begin{cases}
\begin{array}{rl}
\hat{y}(0) = & \! \! \! y_n - \bar{y}_n  \\
\mathcal{H}(\hat{y},\hat{u},\hat{p})= & \! \! \! (0,0,0) \\
\hat{p}(T) - \tilde{\Pi}(\bar{T}-t_n') \hat{y}(T)= & \! \! \! w,
\end{array}
\end{cases}
\end{equation}
where
\begin{align} 
w = \ & \big( \tilde{\Pi}(\bar{T}-t_n')-\Pi(\bar{T}-t_n',Q) \big) (\bar{y}(t_n')-y^\diamond) \notag \\
& \qquad + \big( \tilde{G}(\bar{T}-t_n')-G(\bar{T}-t_n') \big) \tilde{q} + (\tilde{p} - p^\diamond). \label{eq:def_w}
\end{align}
The triple $(\hat{y},\hat{u},\hat{p})$ is the solution to \eqref{eq:optim_sys_zero} with parameters $(y_n-\bar{y}_n,T,\tilde{\Pi}(\bar{T}-t_n'),w)$.
Let us estimate $\| w \|_Y$. By Theorem \ref{theorem:turnpike}, we have
\begin{equation} \label{eq:def_w_1}
\| \bar{y}(t_n') - y^\diamond \|_Y \leq M \big( e^{-\lambda(n \tau + T)} \| y_0 -y_0^\diamond \|_Y + e^{-\lambda(\bar{T}-(n\tau + T))} \| \tilde{q} \| \big).
\end{equation}
By assumption,
\begin{equation} \label{eq:def_w_2}
\| \tilde{G}(\bar{T}-t_n') - G(\bar{T}-t_n',Q) \|_{\mathcal{L}(Y)}
\leq e^{-\lambda(\bar{T} - (n \tau + T))} \| \tilde{G}-G \|_{\infty,\lambda}.
\end{equation}
Combining \eqref{eq:def_w}, \eqref{eq:def_w_1}, and \eqref{eq:def_w_2}, and using the definitions of $K_1$ and $K_2$, we obtain
\begin{equation} \label{eq:estimate_w}
\| w \|_Y \leq e^{-\lambda(n\tau + T)} K_1 + e^{-\lambda(\bar{T}-(n \tau + T))} K_2 + \| \tilde{p} - p^\diamond \|_Y.
\end{equation}
Let us find now some estimates for $(\hat{y},\hat{u},\hat{p})$. To this end, we proceed as in the proof of Theorem \ref{theorem:turnpike}. We consider the solutions $(\hat{y}^{(1)},\hat{u}^{(1)},\hat{p}^{(1)})$ and $(\hat{y}^{(2)},\hat{u}^{(2)},\hat{p}^{(2)})$ to the linear system \eqref{eq:optim_sys_zero}, with parameters $(y_n-\bar{y}_n,T,\tilde{\Pi}(\bar{T}-t_n'),0)$ and $(0,T,\tilde{\Pi}(\bar{T}-t_n'),w)$, respectively, so that $(\hat{y},\hat{u},\hat{p})= (\hat{y}^{(1)},\hat{u}^{(1)},\hat{p}^{(1)}) + (\hat{y}^{(2)},\hat{u}^{(2)},\hat{p}^{(2)})$. Let us first apply Theorem \ref{theo:non_reg_os} to the first system (with $\mu= 0$). We obtain
\begin{equation} \label{eq:estima_hat_1}
\| (\hat{y}^{(1)},\hat{u}^{(1)},\hat{p}^{(1)}) \|_{\Lambda_{T,0}}
\leq M \| y_n - \bar{y}_n \|_Y = M b_n.
\end{equation}
Lemma \ref{lemma:embedding} and
Theorem \ref{theo:non_reg_os}, applied to $(\hat{y}^{(2)},\hat{u}^{(2)},\hat{p}^{(2)})$ with $\mu= -\lambda$, yield
\begin{align}
\| (\hat{y}^{(2)},\hat{u}^{(2)},\hat{p}^{(2)}) \|_{\Lambda_{\tau,0}}
\leq \ & Me^{\lambda \tau} \| (\hat{y}^{(2)},\hat{u}^{(2)},\hat{p}^{(2)}) \|_{\Lambda_{\tau,-\lambda}} \notag \\
\leq \ & Me^{\lambda \tau} \| (\hat{y}^{(2)},\hat{u}^{(2)},\hat{p}^{(2)}) \|_{\Lambda_{T,-\lambda}} \notag \\
\leq \ & M  e^{-\lambda (T-\lambda)} \| w \|_Y.
\label{eq:estima_hat_2}
\end{align}
We deduce from \eqref{eq:estima_hat_1} and \eqref{eq:estima_hat_2} that
\begin{align}
a_n = \ & \max \big( \| \hat{y} \|_{W(0,\tau)}, \| \hat{u} \|_{L^2(0,\tau)} \big) \notag \\
\leq \ & \| (\hat{y},\hat{u},\hat{p}) \|_{\Lambda_{\tau,0}} \notag \\
\leq \ & \| (\hat{y}^{(1)},\hat{u}^{(1)},\hat{p}^{(1)}) \|_{\Lambda_{\tau,0}} + \| (\hat{y}^{(2)},\hat{u}^{(2)},\hat{p}^{(2)}) \|_{\Lambda_{\tau,0}} \notag \\
\leq \ & M \big( b_n + e^{-\lambda(T-\lambda)} \| w \|_Y \big). \label{eq:estimate_for_diff}
\end{align}
Estimate \eqref{eq:estimate_an} follows from \eqref{eq:estimate_w} and \eqref{eq:estimate_for_diff}.
Let us apply again Theorem \ref{theo:non_reg_os} to $(\hat{y}^{(1)},\hat{u}^{(1)},\hat{p}^{(1)})$, now with $\mu= \lambda$. We obtain
\begin{equation} \label{eq:estima_hat_3}
\| (\hat{y}^{(1)},\hat{u}^{(1)},\hat{p}^{(1)}) \|_{\Lambda_{T,\lambda}}
\leq M_0 \| y_n - \bar{y}_n \|_Y = M_0 b_n.
\end{equation}
It follows that $\| \hat{y}^{(1)}(\tau) \|_Y \leq M_0 e^{\lambda \tau} b_n \leq M_0 e^{\lambda \tau_0} b_n = r b_n$.
As a direct consequence of \eqref{eq:estima_hat_2}, we have
$\| \hat{y}^{(2)}(\tau) \|_Y \leq M e^{-\lambda (T-\tau)} \| w \|_Y.$
It follows that
\begin{equation} \label{eq:estimate_for_diff_2}
b_{n+1}= \| \hat{y}(\tau) \|_Y
\leq \| \hat{y}^{(1)}(\tau) \|_Y + \| \hat{y}^{(2)}(\tau) \|_Y
\leq r b_n + M e^{-\lambda(T-\tau)} \| w \|_Y.
\end{equation}
Estimate \eqref{eq:estimate_bn} follows from \eqref{eq:estimate_w} and \eqref{eq:estimate_for_diff_2}.

Let us prove the estimate on $a_{N}$. Denoting by $(y,u,p)$ the solution to \eqref{prob:dyn_prog} with $\theta= N \tau$ and $y_{\theta}= y_{N}$, we obtain that $(\hat{y},\hat{u},\hat{p})(t):= (y,u,p)-(\bar{y},\bar{u},\bar{p})(t_{N} + t)$ is the solution to \eqref{eq:optim_sys_zero}, with parameters $(y_{N}-\bar{y}_{N},\bar{T}-t_{N},Q,0)$.
Applying Theorem \ref{theo:non_reg_os} with $\mu= 0$, we obtain
$a_{N}
\leq \| (\hat{y},\hat{u},\hat{p}) \|_{\Lambda_{\bar{T}-t_{N},0}}
\leq M \| y_{N}-\bar{y}_{N} \|_Y
\leq M b_{N}$,
as was to be proved.

\emph{Step 2: }proof of the general estimates.\\
In order to prove the result, we need to find an estimate for $\sum_{n= 0}^{N} a_n$. We start by estimating $b_n$.
Re-arranging \eqref{eq:estimate_bn}, we obtain that
\begin{align*}
b_{n+1}
\leq \ & r b_n + \big(M e^{-\lambda(T-\tau) - \lambda T} K_1 \big) e^{-n \lambda \tau} \\
& \qquad + \big( M e^{-\lambda(T-\tau) - \lambda(\bar{T}-T)} K_2 \big) e^{n \lambda \tau}
+ \big( M e^{-\lambda(T-\tau)} \| \tilde{p} - p^\diamond \|_Y \big).
\end{align*}
Let us introduce three sequences $(c_n)_{n= 0,...,N}$, $(d_n)_{n= 0,...,N}$, and $(e_n)_{n= 0,...,N}$ defined by $c_0= 0$, $d_0= 0$, $e_0= 0$, and
\begin{align*}
c_{n+1}= r c_n + e^{-n\lambda \tau}, \quad
d_{n+1}= r d_n + e^{n \lambda \tau}, \quad
e_{n+1}= r e_n + 1.
\end{align*}
It is easy to check by induction that
\begin{equation} \label{eq:induc1}
b_n \leq M \big[ \big(e^{-2\lambda T+ \lambda \tau} K_1 \big) c_n
+ \big(e^{-\lambda(T-\tau) - \lambda(\bar{T}-T)} K_2 \big) d_n
+ \big( e^{-\lambda(T-\tau)} \| \tilde{p} - p^\diamond \| \big) e_n \big].
\end{equation}
Lemma \ref{lemma:induction} below allows to estimate $(c_n)_{n= 0,...,N}$, $(d_n)_{n= 0,...,N}$, and $(e_n)_{n= 0,...,N}$. We have $r > e^{-\lambda \tau_0} \geq e^{-\lambda \tau}$, thus
\begin{equation} \label{eq:induc2}
c_n \leq \Big( 1- \frac{e^{-\lambda \tau}}{r} \Big)^{-1} r^{n-1}
\leq \Big( 1- \frac{e^{-\lambda \tau_0}}{r} \Big)^{-1} r^{n-1} \leq M r^{n-1}.
\end{equation}
Moreover, $r < 1 \leq e^{\lambda \tau}$, therefore
\begin{equation} \label{eq:induc3}
d_n \leq \Big( 1 - \frac{r}{e^{\lambda \tau}} \Big)^{-1} e^{(n-1) \lambda \tau}
\leq ( 1 - r )^{-1} e^{(n-1) \lambda \tau}
\leq M e^{(n-1) \lambda \tau}.
\end{equation}
We also have $\| e_n \|_Y \leq M$.
Combining \eqref{eq:induc1}, \eqref{eq:induc2}, and \eqref{eq:induc3}, we obtain that
\begin{equation} \label{eq:induc3_5}
b_n \leq M \big( K_1 e^{-2\lambda T + \lambda \tau} r^{n-1}
+ K_2 e^{-\lambda(T-\tau) - \lambda (\bar{T} - T)} e^{(n-1) \lambda \tau} + e^{-\lambda(T-\tau)} \| \tilde{p} - p^\diamond \| \big).
\end{equation}
Combining the above inequality with \eqref{eq:estimate_an}, we obtain that
\begin{align} 
a_n \leq \ & M  K_1 e^{-\lambda(T-\tau) - \lambda T}  \big( r^{n-1} + e^{-\lambda n \tau} \big) \notag \\
& \quad + M K_2 e^{-\lambda(T-\tau) - \lambda (\bar{T} - T)} \big( e^{(n-1) \lambda \tau} + e^{n \lambda \tau} \big)
+ M e^{-\lambda(T-\tau)} \| \tilde{p} - p^\diamond \|. \label{eq:induc4}
\end{align}
We have
$e^{-n \lambda \tau}
\leq e^{-(n-1) \lambda \tau}
\leq e^{-(n-1) \lambda \tau_0}
\leq r^{n-1}$
as well as $e^{(n-1) \lambda \tau} \leq e^{-\lambda \tau_0} e^{n \lambda \tau}$,
which allows to simplify \eqref{eq:induc4} as follows:
\begin{equation} \label{eq:induc5}
a_n \leq M \big( K_1 e^{-2 \lambda T + \lambda \tau } r^{n-1}
+ K_2 e^{-\lambda(T-\tau) - \lambda (\bar{T} - T)}  e^{n \lambda \tau} + e^{-\lambda(T-\tau)} \| \tilde{p} - p^\diamond \| \big).
\end{equation}
We have
\begin{equation} \label{eq:induc6}
\begin{cases}
\begin{array}{l}
{\displaystyle \sum_{n=0}^{N-1} r^{n-1}
\leq \frac{1}{r} \sum_{n= 0}^{\infty} r^n
= \frac{1}{r(1-r)} \leq M }\\[1.5em]
{\displaystyle \sum_{n=0}^{N-1} e^{n \lambda \tau}
= \frac{e^{N\lambda \tau} - 1}{e^{\lambda \tau}-1}
\leq \frac{e^{N \lambda \tau}}{e^{\lambda \tau_0}-1}
\leq M e^{N \lambda \tau}. }
\end{array}
\end{cases}
\end{equation}
Combining \eqref{eq:induc5} and \eqref{eq:induc6}, we obtain that
\begin{equation} \label{eq:induc8}
\sum_{n=0}^{N-1} a_n
\leq M \big( K_1 e^{-2 \lambda T + \lambda \tau} + K_2 e^{-\lambda(T-\tau) - \lambda(\bar{T} - (N \tau + T))} + N e^{-\lambda(T-\tau)} \| \tilde{p} - p^\diamond \| \big).
\end{equation}
We obtain with \eqref{eq:estimate_an_max} and \eqref{eq:induc3_5} that
\begin{align}
a_{N}
\leq \ & M b_{N} \notag \\
\leq \ & M \big( K_1 e^{-2 \lambda T + \lambda \tau} r^{N-1}
+ K_2 e^{-\lambda(T-\tau) - \lambda (\bar{T} - T)} e^{(N-1) \lambda \tau} + e^{-\lambda(T-\tau)} \| \tilde{p} - p^\diamond \| \big) \notag \\
\leq \ & M \big( K_1 e^{-2 \lambda T + \lambda \tau} + K_2 e^{-\lambda(T-\tau) - \lambda(\bar{T} - (N \tau + T))} + e^{-\lambda(T-\tau)} \| \tilde{p} - p^\diamond \| \big). \label{eq:induc9}
\end{align}
Finally, \eqref{eq:induc8} and \eqref{eq:induc9} yield
\begin{align*}
& \max \big( \| y_{RH}- \bar{y} \|_{W(0,\bar{T})}, \| u_{RH} - \bar{u} \|_{L^2(0,\bar{T};U)} \big) \leq {\textstyle \sum_{n=0}^{N} } a_n \\
& \qquad
\leq M \big( K_1 e^{-2 \lambda T + \lambda \tau} + K_2 e^{-\lambda(T-\tau) - \lambda(\bar{T} - (N \tau + T))} + N e^{-\lambda(T-\tau)} \| \tilde{p} - p^\diamond \| \big),
\end{align*}
which proves \eqref{eq:theo_main_1}.
Using the same techniques as in Lemma \ref{lemma:diff_V}, one can show the existence of $M$ such that
\begin{equation*}
J_{T,Q,q}(u_{RH},y_{RH})- \mathcal{V}_{\bar{T},Q,q}(0,y_0)
\leq M \max \big( \| y_{RH} - \bar{y} \|_{W(0,\bar{T})}, \| u_{RH}- \bar{u} \|_{L^2(0,T;U)} \big)^2.
\end{equation*}
Using then \eqref{eq:theo_main_1}, we obtain \eqref{eq:theo_main_2}. The theorem is proved.
\end{proof}

The following lemma is an independent technical result, used only in the above proof.

\begin{lemma} \label{lemma:induction}
Let $r_1 > 0$ and $r_2 > 0$ be two positive real numbers. Consider the sequence $(\xi_n)_{n \in \mathbb{N}}$ defined by
\begin{equation*}
\xi_0= 0, \quad \xi_{n+1}= r_1 \xi_n + r_2^n, \quad \forall n \in \mathbb{N}.
\end{equation*}
If $r_2 < r_1$, then $\xi_n \leq \frac{1}{1- r_2/r_1} r_1^{n-1}$, for all $n \in \mathbb{N}$.
If $r_1 < r_2$, then $\xi_n \leq \frac{1}{1- r_1/r_2} r_2^{n-1}$, for all $n \in \mathbb{N}$.
\end{lemma}

\begin{proof}
One can easily check by induction that
\begin{equation*}
\xi_n = r_1^{n-1} \sum_{i=0}^{n-1} \Big( \frac{r_2}{r_1} \Big)^i
= r_2^{n-1} \sum_{i=0}^{n-1} \Big( \frac{r_1}{r_2} \Big)^i,
\quad \forall n \in \mathbb{N}.
\end{equation*}
If $r_2 < r_1$, then $\sum_{i=0}^{n-1} \big( \frac{r_2}{r_1} \big)^i \leq \big( 1-\frac{r_2}{r_1} \big)^{-1}$, which proves the first estimate. If $r_1 < r_2$, then $\sum_{i=0}^{n-1} \big( \frac{r_1}{r_2} \big)^i \leq \big( 1-\frac{r_1}{r_2} \big)^{-1}$ and the second estimate follows.
\end{proof}


\section{Infinite-horizon problems} \label{section:infinite_horizon}

\subsection{Formulation of the problem and overtaking optimality}

In this subsection we investigate the case of linear-quadratic optimal control problems with an infinite horizon. The investigated problem can be seen as a limit problem of \eqref{prob:turnpike} when $\bar{T}$ goes to $\infty$.
For this purpose, we introduce the space $L_{\text{loc}}^2(0,\infty)$ of locally square integrable functions and the space $W_{\text{loc}}(0,\infty)$ of functions $y \colon (0,\infty) \rightarrow V$ such that for all $T > 0$, $y_{|(0,T)} \in W(0,T)$.
Consider the problem
\begin{equation*} \label{prob:infinity} \tag{$P(\infty)$}
\begin{cases}
\begin{array}{l}
{\displaystyle
\inf_{\begin{subarray}{c} y \in W_{\text{loc}}(0,\infty) \\ u \in L_{\text{loc}}^2(0,\infty;U) \end{subarray}} \
 \int_0^{\infty} \ell(y(t),u(t)) \dd t } \\[0.5em]
\text{subject to: } \dot{y}(t)= Ay(t) + Bu(t) + f^\diamond, \quad y(0)= y_0.
\end{array}
\end{cases}
\end{equation*}
In general, the above integral is not proper and one needs to use an appropriate notion of optimality. Let us mention that this difficulty would also arise if we chose $W(0,\infty)$ and $L^2(0,\infty;U)$ as function spaces. We call a pair $(y,u) \in W_{\text{loc}}(0,\infty) \times L_{\text{loc}}^2(0,\infty;U)$ \emph{feasible pair} if $\dot{y}= Ay + Bu + f$ and $y(0)= y_0$.

\begin{definition}
A feasible pair $(\bar{y},\bar{u}) \in W_{\text{\emph{loc}}}(0,\infty) \times L_{\text{\emph{loc}}}^2(0,\infty;U)$ is said to be \emph{overtaking} optimal for Problem \eqref{prob:infinity} if for all feasible pairs $(y,u) \in W_{\text{\emph{loc}}}(0,\infty) \times L_{\text{\emph{loc}}}^2(0,\infty;U)$,
$\liminf_{T \to \infty} \big( J_{T,0,0}(u,y) - J_{T,0,0}(\bar{u},\bar{y}) \big) \geq 0$.
\end{definition}

The notion of overtaking optimality is rather classical in the literature, see for example \cite{Z01}, where some existence results are established.
We construct now a pair $(\bar{y},\bar{u})$ which will be the unique overtaking optimal solution to problem \eqref{prob:infinity}.
Let $\tilde{y} \in W(0,\infty)$, $\tilde{p} \in W(0,\infty)$, and $\tilde{u} \in L^2(0,\infty;U)$ be defined by $\dot{\tilde{y}}= A_{\pi} \tilde{y}$, $\tilde{y}(0)= y_0 - y^\diamond$, $\tilde{p}= \Pi \tilde{y}$, $\tilde{u}= - \frac{1}{\alpha} B^*p$.
Using the same arguments as in Lemma \ref{lemma:linSysFiniteHorizon}, we can check that $\tilde{p} \in W(0,\infty)$ with $-\dot{\tilde{p}} = A^* \tilde{p} + C^* C \tilde{y}$. We finally set
\begin{equation*}
(\bar{y},\bar{u},\bar{p})(t)= (y^\diamond,u^\diamond,p^\diamond) + (\tilde{y},\tilde{u},\tilde{p})(t).
\end{equation*}
We have $(\bar{y},\bar{u},\bar{p}) \in W_{\text{loc}}(0,\infty) \times L_{\text{loc}}^2(0,\infty;U) \times W_{\text{loc}}(0,\infty)$. A key point in our analysis is that for all $T>0$, the triplet $(\bar{y},\bar{u},\bar{p})$ is the unique solution to the following optimality system:
\begin{equation} \label{eq:oc_inf}
\bar{y}(0)= y_0, \quad
\mathcal{H}(y,u,p)= (f^\diamond,g^\diamond,-h^\diamond), \quad
\bar{p}(T) - \Pi \bar{y}(T)= p^\diamond - \Pi y^\diamond.
\end{equation}
One can prove with standard arguments $(\bar{y},\bar{u},\bar{p})$ is the unique overtaking optimal solution. We refer the reader to \cite{TR98}, where a more general class of linear-quadratic problems is investigated.

\begin{proposition}
The pair $(\bar{y},\bar{u})$ is the unique overtaking optimal solution to \eqref{prob:infinity}. More precisely, we have
\begin{equation} \label{eq:overtaking_qg}
\liminf_{T \to \infty} \Big( J(u,y,T) - J(\bar{u},\bar{y},T)
- \Big( \frac{\alpha}{2} - \varepsilon \Big) \| u - \bar{u} \|_{L^2(0,T;U)}^2 \Big)
\geq 0,
\end{equation}
for all $\varepsilon > 0$ and for all feasible $(y,u)$.
\end{proposition}

\begin{proof}
Let us first prove that
\begin{equation} \label{eq:overtaking_1}
J(u,y,T) - J(\bar{u},\bar{y},T)
- \frac{\alpha}{2} \int_0^T \|u-\bar{u}\|_U^2 \dd t
= \int_0^T \frac{1}{2} \| C(y-\bar{y}) \|_Z^2 - \langle \bar{p}(T), y(T)- \bar{y}(T) \rangle_Y.
\end{equation}
The calculations are very similar to those of the proof of Lemma \ref{lemma:diff_V}. We have
\begin{align}
J(u,y,T) - J(\bar{u},\bar{y},T) = & \ \int_0^T \frac{1}{2} \| C(y-\bar{y}) \|_Z^2 + \frac{\alpha}{2} \|u- \bar{u}\|_U^2 \dd t \notag \\
& \qquad + \int_0^T \langle C^*C \bar{y} + g, y-\bar{y} \rangle_Y + \langle \alpha \bar{u} + h , u-\bar{u}\rangle_U \dd t.
\label{eq:overtaking_2}
\end{align}
Using $C^*C \bar{y} + g^\diamond= - \dot{\bar{p}} - A^* \bar{p}$, $\alpha \bar{u} + h^\diamond = -B^* \bar{p}$ and integrating by parts, we obtain that
\begin{equation} \label{eq:overtaking_3}
\int_0^T \langle C^*C \bar{y} + g^\diamond, y-\bar{y} \rangle_Y + \langle\alpha \bar{u} + h^\diamond , u-\bar{u}\rangle \dd t
= - \langle \bar{p}(T), y(T)- \bar{y}(T) \rangle_Y.
\end{equation}
Combining \eqref{eq:overtaking_2} and \eqref{eq:overtaking_3}, we obtain \eqref{eq:overtaking_1}.

Let $\hat{y}= y - \bar{y}$. We have $\dot{\hat{y}} = A  \hat{y} + B(u- \bar{u})$, $\hat{y}(0)= 0$.
Therefore, by Lemma \ref{lemma:detectability}, there exists a constant $M$ independent of $T$ such that
\begin{equation*}
\| y(T)- \bar{y} \|_Y = \| \hat{y}(T) \| \leq
M \big( \| u - \bar{u} \|_{L^2(0,T;U)} + \| C(y- \bar{y}) \|_{L^2(0,T;Z)} \big).
\end{equation*}
The adjoint $\bar{p}$ is bounded, since $\bar{p}= p^\diamond + \tilde{p}$, where $\tilde{p} \in W(0,\infty)$.
Therefore,
\begin{align*}
| \langle \bar{p}(T), y(T)- \bar{y}(T) \rangle_Y |
\leq M \big( \| u - \bar{u} \|_{L^2(0,T;U)} + \| C(y- \bar{y}) \|_{L^2(0,T;Z)} \big),
\end{align*}
where again $M$ does not depend on $T$. We deduce that
\begin{align*}
& J(u,y,T) - J(\bar{u},\bar{y},T)
- \Big( \frac{\alpha}{2} - \varepsilon \Big) \| u - \bar{u} \|_{L^2(0,T;U)}^2 \\
& \qquad \geq \big( \varepsilon \| u - \bar{u} \|_{L^2(0,T;U)}^2 - M \| u - \bar{u} \|_{L^2(0,T;U)} \big) \\
& \qquad \qquad + \Big( \frac{1}{2} \| C(y-\bar{y}) \|_{L^2(0,T;Z)}^2 - M \| C(y-\bar{y}) \|_{L^2(0,T;Z)} \Big).
\end{align*}
The two terms on the r.h.s.\@ in the above inequality are bounded from below. Thus, if one of them tends to infinity (which is the case if $\| u - \bar{u} \|_{L^2(0,T;U)} \underset{T \to \infty}{\longrightarrow} \infty$ or $\| C(y-\bar{y}) \|_{L^2(0,T;Z)} \underset{T \to \infty}{\longrightarrow} \infty$), then 
\begin{equation*}
 \Big( J(u,y,T) - J(\bar{u},\bar{y},T)
- \Big( \frac{\alpha}{2} - \varepsilon \Big) \| u - \bar{u} \|_{L^2(0,T;U)}^2 \Big)  \underset{T \to \infty}{\longrightarrow} \infty
\end{equation*}
and therefore, \eqref{eq:overtaking_qg} holds true.
Otherwise, if $\| u - \bar{u} \|_{L^2(0,T;U)}$ and $\| C(y-\bar{y}) \|_{L^2(0,T;Z)}$ are both bounded, then $y - \bar{y} \in W(0,\infty)$ (by Lemma \ref{lemma:detectability}) and therefore $y(T)-\bar{y}(T) \underset{T \to \infty}{\longrightarrow} 0$ (see \cite[Lemma 1]{BreKP17b}). It follows that
$\langle \bar{p}(T), y(T)- \bar{y}(T) \rangle_Y \underset{T \to \infty}{\longrightarrow} 0$.
We deduce then from \eqref{eq:overtaking_1} that
\begin{equation*}
\liminf_{T \to \infty} J(u,y,T) - J(\bar{u},\bar{y},T) - \frac{\alpha}{2} \| u - \bar{u} \|_{L^2(0,T;U)}^2 \geq 0,
\end{equation*}
which proves \eqref{eq:overtaking_qg} and that $(\bar{y},\bar{u})$ is overtaking optimal.

Let us prove uniqueness. Let $(y,u)$ be overtaking optimal. Then, by definition,
\begin{equation*}
0  \leq \liminf_{T \to \infty} \big( J(\bar{u},\bar{y},T) - J(u,y,T) \big)
\leq \limsup_{T \to \infty} \big( J(\bar{u},\bar{y},T) - J(u,y,T) \big).
\end{equation*}
Therefore, using \eqref{eq:overtaking_qg} with $\varepsilon= \frac{\alpha}{4}$,
\begin{align*}
0 \geq \ & - \limsup_{T \to \infty} \big( J(\bar{u},\bar{y},T) - J(u,y,T) \big) \\
= \ & \liminf_{T \to \infty} \big( J(u,y,T) - J(\bar{u},\bar{y},T) \big) \\
\geq \ & \underbrace{\liminf_{T \to \infty} \Big( J(u,y,T) - J(\bar{u},\bar{y},T) - \frac{\alpha}{4} \| u - \bar{u} \|_{L^2(0,T;U)}^2 \Big)}_{\geq 0} + \frac{\alpha}{4} \liminf_{T \to \infty} \| u - \bar{u} \|^2_{L^2(0,T;U)} \\
\geq \ & \frac{\alpha}{4} \liminf_{T \to \infty} \| u - \bar{u} \|_{L^2(0,T;U)}^2.
\end{align*}
We immediately deduce that $u= \bar{u}$. Thus $y= \bar{y}$, which concludes the proof of uniqueness.
\end{proof}

The next lemma deals with the asymptotic analysis of $J(\bar{u},\bar{y},T)$.

\begin{lemma} \label{lemma:asymptotic_cost}
For all $T>0$, the following equality holds true:
\begin{align}
J(\bar{u},\bar{y},T)
= \ & T v^\diamond + \frac{1}{2} \langle y_0 - y^\diamond, \Pi (y_0- y^\diamond) \rangle_Y 
- \langle p^\diamond, \bar{y}(T)-y^\diamond \rangle \notag \\
& \qquad - \frac{1}{2} \langle \bar{y}(T)-y^\diamond, \Pi (\bar{y}(T)-y^\diamond) \rangle_Y, \label{eq:formula_1}
\end{align}
where $v^\diamond$ is the value of problem \eqref{eq:turnpike_pb_setub}.
A direct consequence is the following relation:
\begin{equation} \label{eq:formula_2}
\lim_{T \to \infty} \frac{J(\bar{u},\bar{y},T)}{T} = v^\diamond.
\end{equation}
\end{lemma}

\begin{proof}
A direct consequence of \eqref{eq:oc_inf} is that $(\bar{y},\bar{u})_{|(0,T)}$ is the unique solution to $P(y_0,T,\Pi,q)$, where $q= p^\diamond - \Pi y^\diamond$.
The corresponding $\tilde{q}$ (defined by \eqref{eq:def_tilde_q}) is then
\begin{equation*}
\tilde{q} = q - p^\diamond + \Pi y^\diamond = 0.
\end{equation*}
By Corollary \ref{coro:sensi_for_real_pb}, we have
\begin{equation} \label{eq:formula_3}
\mathcal{V}(y_0,T,\Pi,q)= \mathcal{V}_0(y_0 - y^\diamond, T, \Pi, 0) + \langle p^\diamond, y_0 \rangle_Y + T v^\diamond + \frac{1}{2} \langle y^\diamond, \Pi y^\diamond \rangle_Y + \langle q - p^\diamond, y^\diamond \rangle_Y.
\end{equation}
As was explained in the proof of Corollary \ref{coro:sensi_v_0}, $\mathcal{V}_0(y,T,\Pi,0)= \frac{1}{2} \langle y, \Pi(T,\Pi) y \rangle_Y$. By Lemma \ref{lemma:decay_g}, $\Pi(T,\Pi)= \Pi$. Therefore, \eqref{eq:formula_3} becomes
\begin{equation} \label{eq:formula_4}
\mathcal{V}(y_0,T,\Pi,q)= \frac{1}{2} \langle y_0 -y^\diamond, \Pi(y_0 - y^\diamond) \rangle_Y + \langle p^\diamond, y_0 \rangle_Y + T v^\diamond + \frac{1}{2} \langle y^\diamond, \Pi y^\diamond \rangle_Y - \langle \Pi y^\diamond, y^\diamond \rangle_Y.
\end{equation}
We also have
\begin{align}
J(\bar{u},\bar{y},T)= \ & J(\bar{u},\bar{y},T,\Pi,q) - \frac{1}{2} \langle \bar{y}(T), \Pi \bar{y}(T) - \langle q, \bar{y}(T) \rangle_Y \notag \\
= \ & \mathcal{V}(y_0,T,\Pi,q) - \frac{1}{2} \langle \bar{y}(T), \Pi \bar{y}(T) - \langle q, \bar{y}(T) \rangle_Y. \label{eq:formula_5}
\end{align}
Formula \eqref{eq:formula_1} can be obtained by combining \eqref{eq:formula_4} and \eqref{eq:formula_5}.
Formula \eqref{eq:formula_2} follows from the fact that $\bar{y}-y^\diamond$ converges exponentially to 0.
\end{proof}

\subsection{Analysis of the RHC algorithm}

As before, one can find an approximation of $(\bar{y},\bar{u})$ by using the RHC algorithm.
We have $\bar{p}(T)= \Pi (\bar{y}(T)-y^\diamond) + p^\diamond$. Therefore, a good choice of a terminal cost function in the receding horizon algorithm is a function whose derivative (w.r.t.\@ $y$) is an approximation of $\Pi (y-y^\diamond) + p^\diamond$. We therefore consider
\begin{equation} \label{eq:def_phi_2}
\phi(t,y)= {\textstyle \frac{1}{2}} \langle y- y^\diamond, \hat{\Pi} (y-y^\diamond \rangle_Y + \langle \hat{p}, y \rangle_Y,
\end{equation}
where $\hat{p} \in Y$ and $\hat{\Pi} \in \mathcal{L}(Y)$ is self-adjoint and positive semi-definite.
If one chooses $\hat{\Pi}= \Pi$ and $\hat{p}= p^\diamond$, then the Receding-Horizon algorithm provides the exact overtaking optimal solution to the problem. Let us mention that the function $\phi$ that we propose for the infinite-horizon problem is independent of time.
The Receding-Horizon algorithm is now very similar to Algorithm \ref{algo:rh_turnpike}.

\begin{algorithm}[H]
\begin{algorithmic}
\STATE Input: $\tau \geq 0$, $T \geq \tau$, and $N \in \mathbb{N}$;
\STATE Set $n=0$ and $y_n= y_0$;
\FOR{$n=0,1,2,...,N-1$} \STATE{
Find the solution $(y,u)$ to \eqref{prob:inter} with $\theta= n\tau$, $y_\theta= y_n$, and $\phi$ given by \eqref{eq:def_phi_2};
\STATE Set $y_{RH}(t)= y(t)$ and $u_{RH}(t)= u(t)$ for a.e.\@ $t \in (n \tau, (n+1) \tau)$;
\STATE Set $y_{n+1}= y(\tau)$;
} \ENDFOR
\end{algorithmic}
\caption{Receding-Horizon method}
\label{algo:rh_turnpike_bis}
\end{algorithm}

\begin{theorem} \label{theo:main_bis}
There exist two constants $\tau_0>0$ and $M>0$ such that for all $(y_0,f^\diamond,g^\diamond,h^\diamond)$, for all $\tau_0 \leq \tau \leq T$, the following estimate holds true:
\begin{align}
& \max \big( \| y_{RH}- \bar{y} \|_{W(0,N \tau)}, \| u_{RH} - \bar{u} \|_{L^2(0,N \tau;U)} \big) \notag \\
& \qquad \leq M e^{-\lambda(T-\tau)} \big( e^{-\lambda T} \| \hat{\Pi}- \Pi \|_{\mathcal{L}(Y)} \| y_0 \|_Y + N \| \hat{p} - p^\diamond \|_Y \big). \label{eq:theo_main_1_bis}
\end{align}
\end{theorem}

\begin{remark}
Similar conclusions to the ones for the finite-horizon case can be drawn from the error estimate \eqref{eq:theo_main_1_bis}: reducing $\tau$ and increasing $T$ should improve the quality of the solution obtained with the Receding-Horizon algorithm (still, the case of arbitrarily small values of $\tau$ is not covered). Also, one should choose $\hat{p}= p^\diamond$ since in this case the error estimate becomes independent of $N$.
\end{remark}

\begin{proof}[Proof of Theorem \ref{theo:main_bis}]
Let us fix $\bar{T} > N \tau$.
As a direct consequence of \eqref{eq:oc_inf}, $(\bar{y},\bar{u})_{|(0,\bar{T})}$ is the unique solution to
\eqref{prob:turnpike} with initial condition $y_0$, horizon $\bar{T}$, $Q= \Pi$, and $q= p^\diamond - \Pi y^\diamond$.
The corresponding $\tilde{q}$ is null.
Consider now the pair $(\tilde{y}_{RH},\tilde{u}_{RH})$ obtained when solving this problem with the same values of the parameters $\tau$, $T$, and $N$ and with $\tilde{\Pi}(T)= \hat{\Pi}$ and $\tilde{G}(T)= 0$. By construction, $(y_{RH},u_{RH})$ and $(\tilde{y}_{RH},\tilde{u}_{RH})$ coincide on $(0,N \tau)$. Estimate \eqref{eq:theo_main_1_bis} is directly obtained by applying Theorem \ref{theo:main}. Indeed, the constant $K_2$ involved in \eqref{eq:theo_main_1} is null, since $\tilde{q}= 0$ and since $\sup_{T \in [0,\infty)} \| \tilde{\Pi}(T)- \Pi(T,\Pi) \|_{\mathcal{L}(Y)} =
\| \hat{\Pi} - \Pi \|_{\mathcal{L}(Y)}$,
by Lemma \ref{lemma:decay_g}.
\end{proof}

\section{Numerical verification} \label{section:numerics}

In this section we aim at measuring the tightness of our estimate. Our focus is the dependence of $\| u_{RH}- \bar{u} \|_{L^2(0,\bar{T};U)}$ with respect to $\tau$ and $T$. We consider for this purpose an optimal control problem with state variable of dimension 2 and scalar control, described by the following data:
\begin{align*}
& A= \begin{pmatrix} 0.2 & 0 \\ 0.1 & -0.5 \end{pmatrix}, \quad
B= \begin{pmatrix} 1 \\ 1 \end{pmatrix}, \quad
C= \begin{pmatrix} 1 & 0 \\ 0 & 1 \end{pmatrix}, \quad
\alpha= 0.25, \\
& y_0= \begin{pmatrix} 0 \\ 0 \end{pmatrix}, \quad
Q= \begin{pmatrix} 0 & 0 \\ 0 & 0 \end{pmatrix}, \quad
q= \begin{pmatrix} 0 \\ 0 \end{pmatrix}, \quad
\bar{T}= 30.
\end{align*}
Observe that the matrix $A$ is not stable.
The optimal control and the associated trajectory are represented on the graphs of Figure \ref{figure:graph}. The dashed lines correspond to the values of $u^\diamond$ and $y^\diamond$, respectively.

\begin{figure}[htb]
\centering
\includegraphics[trim= 5cm 9cm 4.5cm 9cm, clip, scale= 0.5]{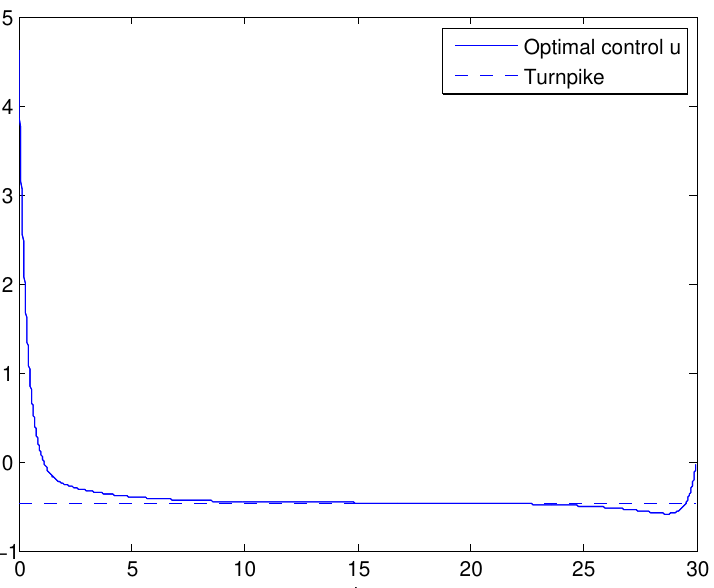}
\includegraphics[trim= 4.5cm 9cm 4.5cm 9cm, clip, scale= 0.5]{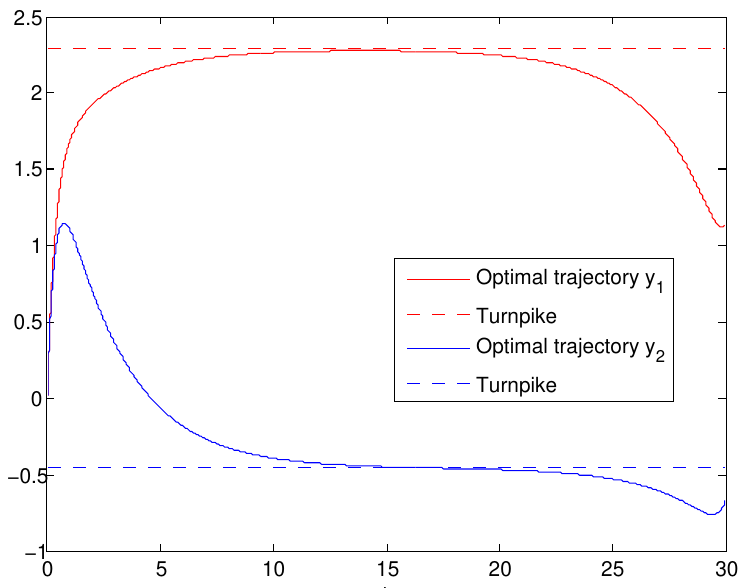}
\caption{Optimal control and optimal trajectory}
\label{figure:graph}
\end{figure}

We have generated different controls with the RHC algorithm, for values of $\tau$ and $T$ ranging from $0.5$ to $7.5$ and with the following parameters:
\begin{equation*}
\tilde{\Pi}= 0, \quad
\tilde{G}= 0, \quad
\tilde{p} = p^\diamond, \quad
N= \lfloor (\bar{T}-2T)/\tau \rfloor.
\end{equation*}
All optimal control problems have been solved with the limited-memory BFGS method, with a tolerance of $10^{-10}$ for the $L^2$-norm of the gradient of the reduced cost function. For the discretization of the state equation, we have used the implicit Euler scheme with time-step equal to $5\times 10^{-3}$.
As a consequence of Theorem \ref{theo:main}, there exist $\tau_0>0$ and $M>0$, both independent of $\tau$ and $T$, such that $\| u_{RH} - \bar{u} \|_{L^2(0,\bar{T})} \leq M e^{-2 \lambda T + \lambda \tau}$, for $\tau_0 \leq \tau \leq T \leq \bar{T}$.
Thus the quantity
\begin{equation*}
\rho(\tau,T):= \ln( \| u_{RH} - \bar{u} \|_{L^2(0,\bar{T})} ) + 2 \lambda T - \lambda \tau
\end{equation*}
is bounded from above, for sufficiently large values of $\tau$.
The results obtained for $\| u_{RH}-\bar{u} \|_{L^2(0,\bar{T})}$ and $100 \rho(\tau,T)$ are shown in Figures \ref{figure:dist} and \ref{figure:rho}, where $\lambda= 0.36$ is the opposite of the spectral absicissa of $A_\pi$. A first observation is that $\| u_{RH}-\bar{u} \|_{L^2(0,\bar{T})}$ is decreasing with respect to $T$ and increasing with respect to $\tau$. Moreover, the number $\rho(\tau,T)$ takes values between $0.40$ and $0.73$. The variation of $\rho(\tau,T)$ can be regarded as small, in comparison with the variation of $2\lambda T-\lambda \tau$ (approximately equal to 5, comparing $T= 0.5$ and $T= 7.5$). We can thus consider that $\rho$ is constant and conclude that our error estimate gives an accurate description of the dependence of $\| u_{RH}- \bar{u} \|_{L^2(0,\bar{T};U)}$ with respect to $\tau$ and $T$.

\begin{figure}[htb]
{\scriptsize
\begin{equation*}
\begin{array}{|c||c|c|c|c|c|c|c|}
\hline
& \multicolumn{7}{c|}{T} \\ \hline
\tau & 0.5 & 1 & 1.5 & 2 & 2.5 & 3 & 3.5\\ \hline
0.5 & 1.7 \,\mathrm{e}{+0} & 1.0 \,\mathrm{e}{+0} & 7.4 \,\mathrm{e}{-1} & 5.3 \,\mathrm{e}{-1} & 3.8 \,\mathrm{e}{-1} & 1.6 \,\mathrm{e}{-2} & 1.1 \,\mathrm{e}{-2} \\
1 & & 1.3 \,\mathrm{e}{+0} & 8.3 \,\mathrm{e}{-1} & 5.9 \,\mathrm{e}{-1} & 4.2 \,\mathrm{e}{-1} & 1.8 \,\mathrm{e}{-2} & 1.2 \,\mathrm{e}{-1} \\
1.5 & & & 1.0 \,\mathrm{e}{+0} & 6.4 \,\mathrm{e}{-1} & 4.4 \,\mathrm{e}{-1} & 1.8 \,\mathrm{e}{-2} & 1.4 \,\mathrm{e}{-2} \\
2 & & & & 8.1 \,\mathrm{e}{-1} & 5.2 \,\mathrm{e}{-1} & 2.2 \,\mathrm{e}{-2} & 1.5 \,\mathrm{e}{-2} \\
2.5 & & & & & 6.6 \,\mathrm{e}{-1} & 2.5 \,\mathrm{e}{-2} & 1.8 \,\mathrm{e}{-2} \\
3 & & & & & & 2.9 \,\mathrm{e}{-2} & 2.1 \,\mathrm{e}{-2} \\
3.5 & & & & & & & 2.4 \,\mathrm{e}{-1} \\ \hline
\end{array}
\end{equation*}
\begin{equation*}
\begin{array}{|c||c|c|c|c|c|c|c|c|}
\hline
& \multicolumn{8}{c|}{T} \\ \hline
\tau & 4 & 4.5 & 5 & 5.5 & 6 & 6.5 & 7 & 7.5 \\ \hline
0.5 & 1.3 \,\mathrm{e}{-1} & 9.3 \,\mathrm{e}{-2} & 6.5 \,\mathrm{e}{-2} & 4.6 \,\mathrm{e}{-2} & 3.2 \,\mathrm{e}{-2} & 2.2 \,\mathrm{e}{-2} & 1.6 \,\mathrm{e}{-2} & 1.1 \,\mathrm{e}{-2} \\
1 & 1.5 \,\mathrm{e}{-1} & 1.0 \,\mathrm{e}{-1} & 7.2 \,\mathrm{e}{-2} & 5.1 \,\mathrm{e}{-2} & 3.5 \,\mathrm{e}{-2} & 2.5 \,\mathrm{e}{-2} & 1.8 \,\mathrm{e}{-2} & 1.2 \,\mathrm{e}{-1} \\
1.5 & 1.5 \,\mathrm{e}{-1} & 1.2 \,\mathrm{e}{-1} & 7.8 \,\mathrm{e}{-2} & 5.3 \,\mathrm{e}{-2} & 4.0 \,\mathrm{e}{-2} & 2.7 \,\mathrm{e}{-2} & 1.8 \,\mathrm{e}{-2} & 1.4 \,\mathrm{e}{-2} \\
2 & 1.9 \,\mathrm{e}{-1} & 1.3 \,\mathrm{e}{-1} & 9.3 \,\mathrm{e}{-2} & 6.2 \,\mathrm{e}{-2} & 4.6 \,\mathrm{e}{-2} & 3.0 \,\mathrm{e}{-2} & 2.2 \,\mathrm{e}{-2} & 1.5 \,\mathrm{e}{-2} \\
2.5 & 2.1 \,\mathrm{e}{-1} & 1.5 \,\mathrm{e}{-1} & 1.1 \,\mathrm{e}{-1} & 7.2 \,\mathrm{e}{-2} & 5.2 \,\mathrm{e}{-2} & 3.5 \,\mathrm{e}{-2} & 2.5 \,\mathrm{e}{-2} & 1.8 \,\mathrm{e}{-2} \\
3 & 2.5 \,\mathrm{e}{-1} & 1.8 \,\mathrm{e}{-1} & 1.2 \,\mathrm{e}{-1} & 8.6 \,\mathrm{e}{-2} & 6.2 \,\mathrm{e}{-2} & 4.1 \,\mathrm{e}{-2} & 2.9 \,\mathrm{e}{-2} & 2.1 \,\mathrm{e}{-2} \\
3.5 & 3.0 \,\mathrm{e}{-1} & 2.1 \,\mathrm{e}{-1} & 1.4 \,\mathrm{e}{-1} & 1.0 \,\mathrm{e}{-1} & 7.2 \,\mathrm{e}{-2} & 4.9 \,\mathrm{e}{-2} & 3.4 \,\mathrm{e}{-2} & 2.4 \,\mathrm{e}{-1} \\
4 & 3.7 \,\mathrm{e}{-1} & 2.5 \,\mathrm{e}{-1} & 1.8 \,\mathrm{e}{-1} & 1.2 \,\mathrm{e}{-1} & 8.4 \,\mathrm{e}{-2} & 5.9 \,\mathrm{e}{-2} & 4.2 \,\mathrm{e}{-2} & 2.8 \,\mathrm{e}{-2} \\
4.5 & & 3.1 \,\mathrm{e}{-1} & 2.1 \,\mathrm{e}{-1} & 1.4 \,\mathrm{e}{-1} & 1.0 \,\mathrm{e}{-1} & 7.0 \,\mathrm{e}{-2} & 4.9 \,\mathrm{e}{-2} & 3.4 \,\mathrm{e}{-1} \\
5 & & & 2.6 \,\mathrm{e}{-1} & 1.7 \,\mathrm{e}{-1} & 1.2 \,\mathrm{e}{-1} & 8.4 \,\mathrm{e}{-2} & 5.9 \,\mathrm{e}{-2} & 4.1 \,\mathrm{e}{-2} \\
5.5 & & & & 2.2 \,\mathrm{e}{-1} & 1.5 \,\mathrm{e}{-1} & 1.0 \,\mathrm{e}{-1} & 7.0 \,\mathrm{e}{-2} & 4.9 \,\mathrm{e}{-1} \\
6 & & & & & 1.8 \,\mathrm{e}{-1} & 1.2 \,\mathrm{e}{-1} & 8.4 \,\mathrm{e}{-2} & 5.8 \,\mathrm{e}{-2} \\
6.5 & & & & & & 1.5 \,\mathrm{e}{-1} & 1.0 \,\mathrm{e}{-1} & 7.0 \,\mathrm{e}{-2} \\
7 & & & & & & & 1.3 \,\mathrm{e}{-1} & 8.5 \,\mathrm{e}{-2} \\
7.5 & & & & & & & & 1.0 \,\mathrm{e}{-1} \\ \hline
\end{array}
\end{equation*}
}
\caption{$\| u_{RH}- \bar{u} \|_{L^2(0,\bar{T})}$ for different values of $\tau$ and $T$.}
\label{figure:dist}
\end{figure}

\begin{figure}[htb]
{\scriptsize
\begin{equation*}
\begin{array}{|c||c|c|c|c|c|c|c|c|c|c|c|c|c|c|c|}
\hline
& \multicolumn{15}{c|}{T} \\ \hline
\tau & 0.5 & 1 & 1.5 & 2 & 2.5 & 3 & 3.5 & 4 & 4.5 & 5 & 5.5 & 6 & 6.5 & 7 & 7.5 \\ \hline
0.5 & 73 & 58 & 61 & 63 & 65 & 66 & 67 & 68 & 69 & 69 & 69 & 70 & 70 & 70 & 70 \\
1 & & 59 & 54 & 56 & 57 & 59 & 60 & 61 & 61 & 62 & 62 & 62 & 63 & 63 & 63 \\
1.5 & & & 54 & 46 & 43 & 53 & 49 & 47 & 55 & 51 & 48 & 57 & 52 & 49 & 57 \\
2 & & & & 51 & 42 & 48 & 43 & 50 & 45 & 51 & 46 & 52 & 46 & 52 & 47 \\
2.5 & & & & & 49 & 40 & 44 & 40 & 43 & 48 & 43 & 46 & 42 & 45 & 49 \\
3 & & & & & & 48 & 40 & 41 & 45 & 40 & 42 & 46 & 42 & 43 & 47 \\
3.5 & & & & & & & 44 & 42 & 43 & 40 & 41 & 43 & 40 & 41 & 42 \\
4 & & & & & & & & 45 & 42 & 42 & 40 & 40 & 41 & 43 & 40 \\
4.5 & & & & & & & & & 45 & 41 & 41 & 42 & 40 & 40 & 41 \\
5 & & & & & & & & & & 47 & 41 & 40 & 40 & 40 & 41 \\
5.5 & & & & & & & & & & & 46 & 41 & 41 & 40 & 40 \\
6 & & & & & & & & & & & & 46 & 41 & 40 & 40 \\
6.5 & & & & && & & & & & & & 46 & 41 & 40 \\
7 & & & & & & & & & & & & & & 46 & 41 \\
7.5 & & & & & & & & & & & & & & & 46 \\ \hline
\end{array}
\end{equation*}
}
\caption{$100(\ln(\| u_{RH}- \bar{u} \|_{L^2(0,\bar{T})}) + 2 \lambda T - \lambda \tau)$, for different values of $\tau$ and $T$.}
\label{figure:rho}
\end{figure}

\section*{Conclusion}

New error bounds for linear optimality systems associated with optimal control problems have been obtained in weighted spaces. They have enabled us to improve the exponential turnpike property for linear-quadratic problems and to obtain a precise error estimate for the control generated by the RHC algorithm.

Future research will be dedicated to the extension of our results to non-linear systems. Let us mention that an error estimate for the RHC method has been obtained for stabilization problems of bilinear systems in \cite{KunP18}, by application of the inverse mapping theorem in weighted spaces. Another axis of research will focus on the extension of our results to the wave equation.

\section*{Acknowledgements}

This project has received funding from the European Research Council (ERC) under the European Union’s Horizon 2020 research and innovation programme (grant agreement No 668998).


\end{document}